\documentclass[a4paper]{article}
\usepackage{blindtext}
\usepackage[utf8]{inputenc}
\usepackage[margin=1in]{geometry}
\usepackage{amsmath}
\usepackage{amssymb}
\usepackage{amsthm}
\usepackage{amscd}
\usepackage[affil-it]{authblk}
\usepackage{enumerate}

\usepackage{graphicx}
\usepackage{amsmath,amsthm,amsfonts,amscd,amssymb,comment,eucal,latexsym,mathrsfs}
\usepackage{stmaryrd}
\usepackage[all]{xy}

\usepackage{epsfig}

\usepackage[all]{xy}
\xyoption{poly}
\usepackage{fancyhdr}
\usepackage{wrapfig}
\usepackage{epsfig}

\theoremstyle{plain}
\newtheorem{thm}{Theorem}[section]
\newtheorem{prop}[thm]{Proposition}
\newtheorem{lem}[thm]{Lemma}
\newtheorem{cor}[thm]{Corollary}

\theoremstyle{definition}
\newtheorem{defn}[thm]{Definition}
\theoremstyle{remark}

\newtheorem{example}{Example}

  \def\C{{\mathbb{C}}}  \def\E{{\mathbb{E}}}         \def\N{{\mathbb{N}}}    \def\R{{\mathbb{R}}}        \def\Z{{\mathbb{Z}}}

\newcommand\Aut{\operatorname{Aut}}
\newcommand\AP{\operatorname{AP}}

\newcommand\id{\operatorname{id}}
\newcommand\IE{\operatorname{IE}}

\newcommand\Map{{\operatorname{Map}}}

\newcommand\Prob{\operatorname{Prob}}

\newcommand\Sym{\operatorname{Sym}}

\newcommand\Span{\operatorname{span}}

\newcommand\topo{\operatorname{top}}

\def\cc{{\curvearrowright}}
\newcommand{\actson}{\curvearrowright}
\newcommand{\actons}{\curvearrowright}

\begin{document}
\begin{titlepage}
\title{Relative Entropy and the Pinsker Product Formula for Sofic Groups }        
\author{Ben Hayes \thanks{The author gratefully acknowledges support by  NSF Grants DMS-1600802, DMS-1827376.}}
\affil{\emph{University of Virginia}\\
         \emph{Kerchof Hall}\\
         \emph{Charlottesville, VA 22904}\\
         \emph{brh5c@virginia.edu}}
\date{\today}
\maketitle

\begin{abstract} We continue our study of the outer Pinsker factor for probability measure-preserving actions of sofic groups. Using the notion of local and doubly empirical convergence developed by Austin we prove that in many cases the outer Pinsker factor of a product action is the product of the outer Pinsker factors. Our results are parallel to those of Seward for Rokhlin entropy. We use these Pinsker product formulas to show that if $X$ is a compact group, and $G$ is a sofic group with $G\cc X$ by automorphisms, then the outer Pinsker factor of $G\cc (X,m_{X})$ is given as a quotient by a $G$-invariant, closed, normal subgroup of $X$. We use our results to show that if $G$ is sofic and $f\in M_{n}(\Z(G))$ is invertible as a convolution operator $\ell^{2}(G)^{\oplus n}\to \ell^{2}(G)^{\oplus n},$ then the action of $G$ on the Pontryagin dual of $\Z(G)^{\oplus n}/\Z(G)^{\oplus n}f$ has completely positive measure-theoretic entropy with respect to the Haar measure.

\end{abstract}
\maketitle


\tableofcontents

\end{titlepage}

\section{Introduction}
	The goal of this paper is to continue the investigation set out in \cite{Me9} on structural properties of the outer Pinsker factor for actions of sofic groups and apply them to the entropy theory of algebraic actions. Recall that for an amenable group $G$ and a probability measure-preserving action $G\cc (X,\mu),$ there is a largest factor $G\cc (Z,\zeta)$ of $G\cc (X,\mu)$ so that $G\cc (Z,\zeta)$ has zero entropy. We call $G\cc (Z,\zeta)$ the \emph{Pinsker factor} of $G\cc (X,\mu).$ We say that $G\cc (X,\mu)$ has completely positive entropy if its Pinsker factor is trivial. For notational purposes, we often write $G\cc (X,\mu)\to G\cc (Y,\nu)$ to mean that $G\cc (Y,\nu)$ is a factor of $G\cc (X,\mu),$ and we will also say that $G\cc (X,\mu)\to G\cc (Y,\nu)$ is an extension in situations where we want to keep track of $X,Y$ and the factor map $X\to Y.$
	
	Since the Pinsker factor is so natural, one expects that it inherits much of the structure that the original action has. For instance, it has been shown, first by \cite{GlasThoWe} if the actions are free and  by \cite{Danilenko} in general, that if $G$ is amenable and $G\cc (X_{j},\mu_{j}),j=1,2$ are two probability measure-preserving actions with Pinsker factors $G\cc (Z_{j},\zeta_{j}),$ then the Pinsker factor of $G\cc (X_{1}\times X_{2},\mu_{1}\otimes \mu_{2})$ is $G\cc (Z_{1}\times Z_{2},\zeta_{1}\otimes \zeta_{2}).$ This shows that products of actions with completely positive entropy have completely positive entropy, but this result has many more applications. For instance, using this it can be shown that if $G$ is an amenable group, if $X$ is a compact group with Haar measure $m_{X},$ and if $G\cc X$ by continuous automorphisms, then there is a closed, $G$-invariant $Y\triangleleft X$ so that the Pinsker factor of $G\cc (X,m_{X})$ is $G\cc (X/Y,m_{X/Y})$ (see \cite{RohlinCPE} and Theorem 8.1 of \cite{ChungLi}). Thus in this case the Pinsker factor inherits the algebraic structure that $G\cc X$ has. Moreover, it is shown in Section 8 of \cite{ChungLi} that this reduces the question as to whether or not an action  of $G$ on a compact, metrizable group $X$ by continuous automorphisms has  completely positive measure-theoretic entropy to whether $G\cc X$ has completely positive topological entropy. As we have already given similar examples in the sofic case of actions with completely positive topological entropy in \cite{Me7}, we wish to carry over the techniques to the sofic world and show that these actions have  completely positive measure-theoretic entropy. Thus in this paper we will give product formulas for Pinsker factors for actions of sofic groups similar to the ones in \cite{GlasThoWe},\cite{Danilenko}.
	
	Entropy for measure-preserving actions of sofic groups was defined in pioneering work of Bowen in \cite{Bow} under the assumption of a generating partition with finite entropy. Work of Kerr-Li in \cite{KLi} removed this assumption and defined topological entropy as well. The class of sofic groups includes all amenable groups, all residually finite groups, all linear groups, all residually sofic groups, all locally sofic groups, and is closed under free products with amalgamation over amenable subgroups (see \cite{Pan},\cite{DKP},\cite{PoppArg}). Thus sofic entropy is a vast generalization of entropy for amenable groups as defined by Kieffer in \cite{Kieff}. Since we will need to refer to it later, we roughly describe the definition of soficity and sofic entropy. Roughly, $G$ is sofic if there is a sequence of functions (not assumed to be homomorphisms) $\sigma_{i}\colon G\to S_{d_{i}}$ which give ``almost free almost actions.'' By ``almost action" one just means that for each $g,h\in G$ the set of points $1\leq j\leq d_{i}$ for which the action hypothesis $\sigma_{i}(gh)(j)=\sigma_{i}(g)\sigma_{i}(h)(j)$ fails has very small size as $i\to\infty,$ and by ``almost free" one means that for all $g\in G\setminus\{e\}$ and ``most" $1\leq j\leq d_{i}$ we have $\sigma_{i}(g)(j)\ne j.$ Given a probability measure-preserving action $G\cc (X,\mu),$ the sofic entropy of $G\cc (X,\mu)$ (with respect to $(\sigma_{i})_{i}$) measures the exponential growth rate as $i\to\infty$ of ``how many" finitary approximations $\phi\colon\{1,\dots,d_{i}\}\to X$ of $G\cc (X,\mu)$ there are which are compatible with this sofic approximation. We call such approximations ``microstates.'' Analogous to the definition for amenable groups, one can define the Pinsker factor for actions of sofic groups (we remark that the Pinsker factor depends on the sofic approximation).
	
	Though this definition of entropy ends up being satisfactory for many purposes, there are properties of entropy for actions of amenable groups (e.g. decrease of entropy under factor maps) which necessarily fail for actions of sofic groups. There are examples due to Ornstein and Weiss which show that, under any reasonable definition of entropy for actions of nonamenable groups, entropy will increase under certain factor maps. Thus there could be factors of the Pinsker factor which have positive entropy. One can ``fix'' this by considering \emph{entropy in the presence}. Implicit in work of Kerr in \cite{KerrPartition}, entropy in the presence measures for a given factor $G\cc (Y,\nu)$ of $G\cc (X,\mu)$ ``how many'' finitary approximations of $G\cc (Y,\nu)$ there are which ``lift'' to finitary approximations of $G\cc (X,\mu).$ If $G\cc (Z,\zeta)$ is a factor of $G\cc (Y,\nu),$ then the entropy of $G\cc (Z,\zeta)$ in the presence of $G\cc (X,\mu)$ is at most the entropy of $G\cc (Y,\nu)$ in the presence of $G\cc (X,\mu)$. Entropy in the presence leads us to define the \emph{outer Pinsker factor} which is the largest factor $G\cc (Y,\nu)$ of $G\cc (X,\mu)$ which has zero entropy in the presence of $G\cc (X,\mu)$ with respect to $(\sigma_{i})_{i}.$ We remark again that the definition of the outer Pinsker factor depends on the choice of a sofic approximation. In our opinion, the outer Pinsker factor is a more natural object and it is this version of the Pinsker factor we give a product formula for. We do not know of general conditions under which  a product formula holds for the (non-outer) Pinsker factor. We remark that there are counterexamples for the Pinsker product formula (for either outer or non-outer Pinsker factors), but none of the known counterexamples are entirely satisfactory. All currently known counterexamples to a Pinsker product formula involve something akin to considering actions $G\cc (X,\mu),G \cc (Y,\nu),$ so that $G\cc (X,\mu),G \actson (Y,\nu)$ have microstates with respect to $(\sigma_{i})_{i},$ but $G \actson (X\times Y,\mu\otimes \nu)$ does not admit microstates with respect to $(\sigma_{i})_{i}.$ It would be interesting to find an example where $G\cc (X\times Y,\mu\otimes \nu)$ has microstates with respect to $(\sigma_{i})_{i},$ but for which the Pinsker factor of $G\actson (X\times Y,\mu\otimes \nu)$ is not the product of the  Pinsker factors of $G \actson (X,\mu),G \actson (Y,\nu),$ or which has the analogous property for outer Pinsker factors.
	
	Even once we consider the correct version of the Pinsker factor, there are still delicate issues that occur in the investigation of product actions for sofic entropy. As mentioned in the preceding paragraph, even if one assumes that we have microstates $\phi_{j}\colon \{1,\dots,d_{i}\}\to X_{j},j=1,2$ for $G\cc (X_{j},\mu_{j}),j=1,2,$ there is  still no way to ensure that $G\cc (X_{1}\times X_{2},\mu_{1}\otimes \mu_{2})$ has microstates. Because of this, even if $G\cc (X_{j},\mu_{j}),j=1,2$ have completely positive entropy, it may still be the case that $G\cc (X_{1}\times X_{2},\mu_{1}\otimes \mu_{2})$ has entropy $-\infty.$  To deal with this, we use the notion of local and doubly empirical convergence developed by Austin in \cite{AustinAdd}. We briefly describe local and doubly empirical convergence. Suppose that $G\cc (X,\mu)$ is a probability measure-preserving action where $X$ is compact, $\mu$ is a completed Borel probability measure, and $G\cc X$ by homeomorphisms. Given  a sequence $\mu_{i}\in \Prob(X^{d_{i}}),$ the assertion that $\mu_{i}$ locally and doubly empirically converges to $\mu$ is  a combination of three assumptions. The first is that $\mu_{i}$  is mostly supported, asymptotically as $i\to\infty,$ on the space of microstates. Secondly, we require that $\mu_{i}$ has the property that for every $f\in C(X)$ and for ``most'' $j$ in $\{1,\dots,d_{i}\}$ we have
	\[\int_{X^{d_{i}}}f(x(j))\,d\mu_{i}(x)\approx \int_{X}f\,d\mu.\]
Lastly, we require that $\mu_{i}\otimes \mu_{i}$ is also almost supported on the space of microstates for $X\times X.$ One of the main results of \cite{AustinAdd} is that the existence of a sequence of measures $\mu_{i}$ with $\mu_{i}\to^{lde}\mu$ implies that there is a way to produce (at random) a microstate for $G\cc (X\times Y,\mu\otimes \nu)$ for any other action $G\cc (Y,\nu)$ (assuming $G\cc (Y,\nu)$ has microstates to begin with). We thus define $G\cc (X,\mu)$ to be \emph{strongly sofic} (with respect to $(\sigma_{i})_{i}$) if there is some compact model for $G\cc (X,\mu)$  so that there is a sequence $\mu_{i}\in \Prob(X^{d_{i}})$ which locally and doubly empirically converges to $\mu.$ By \cite[Corollary 5.18]{AustinAdd}, this does not depend upon the choice of compact model for $G\cc (X,\mu).$ Strong soficity ends up being a crucial property which allows us to prove a product formula for Pinsker factors.

\begin{thm}\label{T:PinskerProdIntro} Let $G$ be a countable, discrete, sofic group with sofic approximation $\sigma_{i}\colon G\to S_{d_{i}}.$ Let $(X_{j},\mu_{j}),j=1,2$ be Lebesgue probability spaces and $G\cc (X_{j},\mu_{j}),j=1,2$ probability measure-preserving actions which are strongly sofic with respect to $(\sigma_{i})_{i}.$ Let $(Z_{j},\zeta_{j})$ be the outer Pinsker factor of $G\cc (X_{j},\mu_{j}),j=1,2$ with respect to $(\sigma_{i})_{i}.$ Then the outer Pinsker factor of $G\cc (X_{1}\times X_{2},\mu_{1}\otimes \mu_{2})$ with respect to $(\sigma_{i})_{i}$ is $G\cc (Z_{1}\times Z_{2},\zeta_{1}\otimes \zeta_{2}).$
\end{thm}

We remark that it follows from \cite[Proposition 8.4]{AustinAdd} that $G\actson (X,\mu)$ is strongly sofic if and only if every microstate $\phi\colon \{1,\dots,d_{i}\}\to Y$ for another action $G\actson (Y,\nu)$ lifts to a microstate $\psi\colon \{1,\dots,d_{i}\}\to X\times Y.$ In this sense, the assumption of strong soficity of $G\actson (X,\mu)$ is natural for the existence of a Pinsker product formula (for example, it is currently the only way to guarantee that $G\actson (X\times Y,\mu\otimes \nu)$ has positive entropy if one of $G\actson (X,\mu),G\actson (Y,\nu)$ do).

We in fact prove something more general than Theorem \ref{T:PinskerProdIntro}, namely we prove a product formula for ``relative outer Pinsker factors''. We refer the reader to Corollary \ref{C:PinskerProd} for the precise statement.  One can also consider Rokhlin entropy as investigated by Seward in \cite{SewardKrieger} and in a similar manner define the outer Rokhlin Pinsker factor (see \cite{SewardPinskerProd} for the definition). Seward has shown in \cite{SewardPinskerProd} that if $G\cc (X_{j},\mu_{j})$ are two free, probability measure-preserving actions, then the outer Rokhlin Pinsker factor of $G\cc (X_{1}\times X_{2},\mu_{1}\otimes \mu_{2})$ is the product of the outer Rokhlin Pinsker factors for $G\cc (X_{1},\mu_{1}),G\cc (X_{2},\mu_{2}),$ assuming that each $G\cc (X_{j},\mu_{j}),j=1,2$ is weakly contained in a Bernoulli shift action.  Motivated by these results, we prove the following permanence properties of strong soficity.

\begin{thm}\label{T:strongsoficpermanenceintro}  Let $G$ be a countable, discrete, sofic group with sofic approximation $\sigma_{i}\colon G\to S_{d_{i}}.$
\begin{enumerate}[(i)]
\item If $G\cc (X,\mu)$ is a probability measure-preserving action with $(X,\mu)$ Lebesgue, and if $G\cc (X,\mu)$ is an inverse limit of actions which are strongly sofic with respect to $(\sigma_{i})_{i},$ then $G\actson (X,\mu)$ is strongly sofic with respect to $(\sigma_{i})_{i}.$
\item Suppose that $G\cc (Y,\nu),G\cc (X,\mu)$ are two probability measure-preserving actions with $(X,\mu),(Y,\nu)$ Lebesgue, and that $G\cc (Y,\nu)$ is weakly contained in $G\cc (X,\mu).$ If $G\cc (X,\mu)$ is strongly sofic with respect to $(\sigma_{i})_{i}$, then so is $G\cc (Y,\nu).$
\end{enumerate}

\end{thm}
Theorems \ref{T:strongsoficpermanenceintro} and \ref{T:PinskerProdIntro} give us a Pinsker product formula for actions weakly contained in Bernoulli shift actions, and so our results are parallel to those of Seward in \cite{SewardPinskerProd}. To prove Theorem \ref{T:strongsoficpermanenceintro}, we rephrase local and doubly empirical convergence in functional analytic terms in a way which avoids specifying a dynamically generating pseudometric.

Following arguments of \cite{Berg},\cite{ChungLi} we can prove that if $X$ is a compact, metrizable group,  if $G\actson X$ by continuous automorphisms, and if $G\actson (X,m_{X})$ is strongly sofic with respect to a fixed sofic approximation of $G,$ then the outer Pinsker factor is given by the action on a quotient of $X$ by a $G$-invariant, closed, normal subgroup.

\begin{thm}\label{T:algintro} Let $G$ be a countable, discrete, sofic group with sofic approximation $\sigma_{i}\colon G\to S_{d_{i}}.$ Let $X$ be a compact, metrizable group with $G\actson X$ by automorphisms. If $G\actson (X,m_{X})$ is strongly sofic with respect to $(\sigma_{i})_{i}$, then there exists a closed, normal, $G$-invariant subgroup $Y\subseteq X$ so that the outer Pinsker factor (with respect to $(\sigma_{i})_{i}$) is given by the map $X\to X/Y.$
\end{thm}

In \cite{Me9} we gave a formula, analogous to the Kerr-Li formulation of measure-theoretic entropy in \cite{KLi}, for measure-theoretic entropy in the presence in terms of a topological model for the factor $G\cc (X,\mu)\to G\cc (Y,\nu)$ (see \cite{Me9} Definition 2.7). Li-Liang then gave a similar formulation for topological entropy in the presence  (see \cite[Definition 9.3]{LiLiang2}). Both of these definitions are recalled in Definitions \ref{D:topentrpresence} and \ref{D:measentpresence} of this paper. Topological entropy in the presence is analogous to measure-theoretic entropy in the presence in that it measures, for a given topological factor $G\cc Y$ of $G\cc X,$ how many microstates $G\cc Y$ has which ``lift''' to microstates for $G\cc X.$  In Theorem 1.1 of \cite{Me12} we related topological entropy in the presence to measure-theoretic entropy in the presence for strongly sofic actions on compact groups. To do this, we used the similarity between our definition of measure-theoretic entropy in the presence in terms of a given compact model, and Li-Liang's definition of topological entropy in the presence. If $X$ is a compact, metrizable space and $G\cc X$ by homeomorphisms, then we say that $G\cc X$ has completely positive topological entropy in the presence if the topological entropy of $G\cc Y$ in the presence of $G\cc X$ is positive whenever $Y$ is a nontrivial topological factor of $X$. As with measure-theoretic entropy, completely positive topological entropy in the presence implies completely positive topological entropy. Using Theorem 1.1 of \cite{Me12}, we have the following corollary of Theorem \ref{T:algintro} connecting completely positive \emph{topological} entropy in the presence to completely positive \emph{measure-theoretic} entropy in the presence.

\begin{cor} Let $G$ be a countable, discrete, sofic group with sofic approximation $\sigma_{i}\colon G\to S_{d_{i}}.$ Let $X$ be a compact, metrizable group with $G\actson X$ by automorphisms. If $G\actson (X,m_{X})$ is strongly sofic with respect to $(\sigma_{i})_{i}$, then the following are equivalent:
\begin{enumerate}[(i)]
\item $G\cc (X,m_{X})$ has completely positive measure-theoretic entropy in the presence with respect to $(\sigma_{i})_{i}$, \label{I:algcpmeasentpres}
\item $G\cc X$ has completely positive topological entropy in the presence with respect to $(\sigma_{i})_{i}$, \label{I:algcptopentpres}
\item for any closed, normal, $G$-invariant subgroup $Y\subseteq X$  with $Y\ne X$ the topological entropy of $G\cc X/Y$ in the presence of $G\cc X$ (with respect to $(\sigma_{i})_{i}$) is positive.  \label{I:algtopentsubgrp}
\end{enumerate}
\end{cor}

 	  We now consider specific actions on compact groups. For a countable group $G$ and  $f\in M_{m,n}(\C(G)),$ write $f_{sl}=\sum_{x\in G}\widehat{f_{sl}}(x)x$ for $1\leq s\leq m,1\leq l\leq n.$  We define $\lambda(f)\colon \ell^{2}(G)^{\oplus n}\to \ell^{2}(G)^{\oplus m}$ and $r(f)\colon \ell^{2}(G)^{\oplus m}\to \ell^{2}(G)^{\oplus n}$ by:
 	 \[(\lambda(f)\xi)(l)(g)=\sum_{1\leq s\leq m}\sum_{x\in G}\widehat{f_{ls}}(x)\xi(s)(x^{-1}g),\mbox{ for $1\leq l\leq m,g\in G,$}\]
 	 \[(r(f)\xi)(l)(g)=\sum_{1\leq s\leq n}\sum_{x\in G}\widehat{f_{sl}}(x)\xi(s)(x^{-1}g),\mbox{ for $1\leq l\leq n,g\in G.$}\]
We denote by $X_{f}$ the Pontryagin dual of $\Z(G)^{\oplus n}/r(f)(\Z(G)^{\oplus m}).$ We have an action $G\cc X_{f}$ given by
\[(gx)(a)=x(g^{-1}a),\mbox{ $x\in X_{f},$ $a\in \Z(G)^{\oplus n}/r(f)(\Z(G)^{\oplus m})$}.\]
In particular, by \cite{Me5} (see Proposition 2.15 of \cite{Me12}) we know that Theorem \ref{T:PinskerProdIntro} applies to $G\cc  (X_{f},m_{X_{f}})$ when $f\in M_{n}(\Z(G))$ and $\lambda(f)$ is injective. We combine this with our previous results in \cite{Me7} to give examples of algebraic actions which have completely positive measure-theoretic entropy in the presence (i.e. their outer Pinsker factor is trivial). Note that completely positive measure-theoretic entropy in the presence implies completely positive measure-theoretic entropy.

\begin{cor}\label{C:algintroCPE} Let $G$ be a countable, discrete, sofic group with sofic approximation $\sigma_{i}\colon G\to S_{d_{i}}.$ Let $f\in M_{n}(\Z(G)),$ and suppose that $\lambda(f)$ is invertible. Then $G\cc (X_{f},m_{X_{f}})$ has completely positive entropy in the presence. That is, if $G\cc (Y,\nu)$ is a measure-theoretic factor of $G\cc (X_{f},m_{X_{f}})$ and $\nu$ is not a point mass, then the entropy of $G\cc (Y,\nu)$ in the presence of $G\cc (X_{f},m_{X_{f}})$ (with respect to $(\sigma_{i})_{i}$) is positive.

\end{cor}
We remark that Kerr in \cite{KerrCPE} showed that Bernoulli actions have completely positive entropy. Our result covers his, since it follows from \cite{BowenOrn} that every Bernoulli action is a factor of an algebraic action of the above form. The above result is of interest to us because it furthers the connections between the ergodic theoretic properties of $G\cc (X_{f},m_{X_{f}})$ and the operator theoretic properties of $\lambda(f)$ as shown in e.g. \cite{Den},\cite{Li2},\cite{LiSchmidtPet},\cite{LiThom}. Another interesting aspect of Corollary \ref{C:algintroCPE} is that it has long been asked whether $G\cc (X_{f},m_{X_{f}})$ is Bernoulli if $f\in M_{n}(\Z(G))$ is invertible as an operator $\ell^{1}(G)^{\oplus n}\to \ell^{1}(G)^{\oplus n}$ (see e.g. \cite{LindSchmidt2} Conjecture 6.8). We believe that we should in fact have many Bernoulli-like properties of $G\cc (X_{f},m_{X_{f}})$ when $\lambda(f)$ is invertible (i.e. $f$ is invertible as a convolution operator $\ell^{2}(G)^{\oplus n}\to \ell^{2}(G)^{\oplus n}$). It is easy to see that being invertible as an operator $\ell^{1}(G)^{\oplus n}\to \ell^{1}(G)^{\oplus n}$ implies invertibility as an operator $\ell^{2}(G)^{\oplus n}\to \ell^{2}(G)^{\oplus n}.$ The above corollary provides further evidence of the Bernoulli-like behavior of these actions by showing that they have completely positive entropy. Another result on completely positive entropy for actions of sofic groups is the work of Austin-Burton in \cite{AustinBurtonCPE}. They show that if $G$ has an element of infinite order, then it has continuum many actions with completely positive entropy, none of which factor onto each other. Thus any group with an element of infinite order has many actions with completely positive entropy.

We mention here a few examples of $f\in \Z(G)$ which have $\lambda(f)$ invertible. First, if $f$ is invertible in the convolution algebra $\ell^{1}(G),$ then $\lambda(f)$ is invertible. By standard Banach algebra arguments, this applies for example if
\[f=b-\sum_{x\in G}a_{x}x\]
for $(a_{x})_{x\in G}\in c_{c}(G,\Z)$ and $b\in \Z$ with $\sum_{x\in G}|a_{x}|<|b|.$ In the case that $f$  is invertible in $M_{n}(\ell^{1}(G)),$ Corollary \ref{C:algintroCPE} follows from Theorem \ref{T:algintro}, Theorem 1.1 of \cite{Me12}, and Theorem 6.7 of \cite{KerrLi2}. However, if we only assume that $\lambda(f)$ is invertible (and not that $f$ is invertible in $M_{n}(\ell^{1}(G))$), then we need to use the full strength of \cite{Me7} instead of Theorem 6.7 of \cite{KerrLi2}.

Now consider $f$ as above, but now suppose  $b\ne 0,$ and that
\[\sum_{x\in G}|a_{x}|=b.\]
Then, if $\{y^{-1}x:x,y\in G,a_{x}\ne 0,a_{y}\ne 0\}$ generates a nonamenable group, it is well known that
\[\left\|\sum_{x}a_{x}\lambda(x)\right\|<b.\]
This again (by standard Banach algebra arguments) implies that $\lambda(f)$ is invertible. We do not know if there are certain choices of $a_{x}$ in this case which make $f$ invertible in $\ell^{1}(G).$ However, if $a_{x}\geq 0$ for every $x\in G,$ then one can show that $f$ is not invertible in $\ell^{1}(G)$ in this example (this follows from the same argument as Theorem A.1 of \cite{BowenLi}). If $a_{x}\geq 0$ for every $x\in G,$ this example is called the \emph{harmonic model}, as $X_{f}$ in this case may be regarded as the space of $\mu$-harmonic functions $f\colon G\to \R/\Z$ (i.e. functions with $\mu*f=bf$) where $\mu$ is the measure $\sum_{x\in G}a_{x}\delta_{x}.$ The entropy theory of the harmonic model for nonamenable $G$ was first studied in \cite{BowenLi}, and  is related to wired spanning forests and tree entropy as defined by Lyons in \cite{Lyons}. Similar examples can be given by considering
\[f=b+\sum_{x\in G}a_{x}x.\]
Suppose that there exist $g,h\in G$  so that the semigroup generated by $g,h$ (but not necessarily the group generated by $g,h$) is a free semigroup. Then by Example A.1 of \cite{Li2} we know that
\[f=3e+(e-g-g^{2})h\]
has $\lambda(f)$ invertible, but is not invertible in $\ell^{1}(G).$

We make a few brief remarks on the proof of Theorem \ref{T:PinskerProdIntro}. Important in the proof is the new notion of \emph{relative sofic entropy}. Given an extension $G\cc (X,\mu)\to G\cc (Y,\nu),$ relative entropy roughly measures the maximal number of ways there are to ``lift'' any fixed microstate for $Y$ to one for $X.$ Note that this is different than entropy in the presence, which roughly measures ``how many'' microstates there are for $Y$ which have a ``lift'' to $X.$  We show in Appendix \ref{A:Agree} that this agrees with relative entropy when the group acting is amenable. The method involves a language translation of results of Paunescu, Popa in \cite{Pan},\cite{PoppArg} into an ultrafilter-free form, as well as the main results of \cite{BowenAmen},\cite{KLi2}.  By the Appendix of \cite{Me9}, when the acting group is amenable  the entropy of $Y$ in the presence of $X$ is just the entropy of $Y$ (and is thus \emph{not} the relative entropy). Relative entropy is defined in Section \ref{S:relent} and its main properties are established. These properties allows us to define the \emph{relative outer Pinsker factor of $Y$ relative to $Z$ in the presence of $X$} for extensions
\[G\cc (X,\mu)\to G\cc (Y,\nu)\to G\cc (Z,\zeta).\]
We then proceed to follow the methods of Glasner-Thouvenot-Weiss in \cite{GlasThoWe} to prove our Pinsker product formula. We make a minor modification to these methods by noting that every action $G\cc (Y,\nu)$ can be written as a factor of an action  with a large automorphisms group, and that we can take this action to be strongly sofic if $G\cc (Y,\nu)$ is.  Strong soficity ends up being crucial in this argument in the following way: if $G\cc (Y,\nu)$ is strongly sofic, and $G\cc (Z,\zeta)$ is a factor of $G\cc (X,\mu),$ then the entropy of $G\cc (X\times Y,\mu\otimes \nu)$ relative to $G\cc (Z\times Y,\zeta\otimes \nu)$ (in the presence of $X\times Y$) is the entropy of $G\cc (X,\mu)$ relative to $G\cc (Z,\zeta).$ This is the crucial step (along with the construction of an extension with large automorphism group) in the argument that the outer Pinsker factor of $X\times Y$ relative to $Y$ is the product of the outer Pinsker factor of $X$ and the whole system $Y.$

	We make some brief comments on the organization of the paper. In Section \ref{S:relent}, we define relative entropy for actions of sofic groups and prove its basic properties. We also define the relative outer Pinsker algebra in this section. Because of the duality between factors and sigma-algebras this also defines the Pinsker factor. We will prefer to (mostly) state the results in the paper in terms of sigma-algebras as it makes the results clearer and avoids any issues with the fact that the Pinsker factor is only well-defined up to isomorphism. In Section \ref{S:dq}, we give some preliminaries on local and doubly empirical convergence and state the definition of strong soficity. In this section we also prove Theorem \ref{T:strongsoficpermanenceintro}.  We then prove that  if $G\cc (Y,\nu)$ is strongly sofic, then the entropy of $G\cc (X\times Y,\mu\otimes \nu)$ relative to $G\cc (Z\times Y,\zeta\otimes \nu)$ is the entropy of $G\cc (X,\mu)$ relative to $G\cc (Z,\zeta)$. In Section \ref{S:PinskerProd}, we follow the methods in \cite{GlasThoWe} and prove Theorem \ref{T:PinskerProdIntro}. In Section \ref{S:algebraic}, we give applications to actions on compact groups by automorphisms including the proof of Theorem \ref{T:algintro} and Corollary \ref{C:algintroCPE}.

\textbf{Acknowledgments.} I am grateful to Lewis Bowen and Brandon Seward for invaluable conversations. I would like to thank Brandon Seward for suggesting a major simplification of the proof of Lemma \ref{L:ultraproductrick}. Part of this work was inspired by discussion at the ``Measured Group Theory'' conference at the Erwin Schr\"{o}dinger Institute. I thank the Erwin Schr\"{o}dinger Institute for its hospitality and providing a stimulating environment in which to work. I thank the anonymous referee, whose numerous comments and simplifications greatly improved the paper.

\section{Relative Sofic Entropy}\label{S:relent}

We begin by recalling the definition of a sofic group. For a set $A$ and $n\in\N,$ we identify $A^{n}$ with all functions $\{1,\dots,n\}\to A.$ For a finite set $A,$ we use $u_{A}$ for the uniform measure on $A.$ If $A=\{1,\dots,n\}$ for some $n\in\N,$ then we use $u_{n}$ instead of $u_{\{1,\dots,n\}}.$ We use $S_{n}$ for the group of all bijections $\{1,\dots,n\}\to \{1,\dots,n\}.$
\begin{defn}\label{D:sofic} Let $G$ be a countable, discrete group. A sequence of maps $\sigma_{i}\colon G\to S_{d_{i}}$ is said to be a \emph{sofic approximation} if:
\begin{itemize}
\item for every $g,h\in G$ we have $u_{d_{i}}(\{j:\sigma_{i}(g)\sigma_{i}(h)(j)=\sigma_{i}(gh)(j)\})\to 1,$ and
\item for every $g\in G\setminus\{e\}$ we have $u_{d_{i}}(\{j:\sigma_{i}(g)(j)=j\})\to 1.$
\end{itemize}
We say that $G$ is \emph{sofic} if it has a sofic approximation.

\end{defn}

Examples of sofic groups include all amenable groups, all residually finite groups, all linear groups, all residually sofic groups, and all locally sofic groups. The class of sofic groups is also closed under free products with amalgamation over amenable subgroups (see \cite{Pan},\cite{DKP},\cite{PoppArg}).

Throughout the paper, we use the convention that a Lebesgue probability space is a probability space $(X,\mathcal{X},\mu)$ which is isomorphic modulo null sets to a \emph{completion} of a probability space $(Y,\mathcal{Y},\nu),$ where $(Y,\mathcal{Y})$ is a standard Borel space. In essentially every setting for this paper, probability spaces will be complete and we will adopt notational conventions to account for this. For example, if $(X,\mathcal{X},\mu)$ is a complete probability space, and $(A_{\alpha})_{\alpha\in I}$ are complete sub-sigma-algebras, we use $\bigvee_{\alpha\in I}A_{\alpha}$ for the smallest \emph{complete} sub-sigma-algebra of $\mathcal{X}$ containing all the $A_{\alpha}.$ Similarly, if $(X,\mathcal{X},\mu),(Y,\mathcal{Y},\nu)$ are two complete probability spaces, we will use $\mathcal{X}\otimes \mathcal{Y}$ for the completion of the sigma-algebra generated by $\{A\times B:A\in \mathcal{X},B\in \mathcal{Y}\}$ with respect to the usual product measure. We will also use $\mu\otimes \nu$ for the completion of the usual product measure.

Recall that if $(X,\mu)$ is a Lebesgue probability space, and $A$ is a measurable space, then a measurable map $\alpha\colon X\to A$ is called an \emph{observable}. We say that $\alpha$ is \emph{finite} if $A$ is a finite set and all subsets of $A$ are measurable. If $\mathcal{S}$ is a subalgebra (not necessarily a sub-sigma-algebra) of measurable subsets of $X,$ we say that $\alpha$ is $\mathcal{S}$-measurable if $\alpha^{-1}(\{a\})\in \mathcal{S}$ for all $a\in A.$  If $G$ is a countable, discrete group acting by probability measure-preserving transformations on $X$ and $F$ is a finite subset of $G,$  we define
$\alpha^{F}\colon X\to A^{F}$
by
\[\alpha^{F}(x)(h)=\alpha(h^{-1}x) \mbox{ for all $h\in F,x\in X.$}\]
If $\sigma\colon G\to S_{d}$ is a function (not assumed to be a homomorphism) and $\phi\in A^{d},$ we define
$(\phi_{\sigma}^{F})\colon\{1,\dots,d\}\to A^{F}$
by
\[(\phi_{\sigma}^{F})(j)(g)=\phi(\sigma(g)^{-1}(j)).\]
 If $\Omega$ is a set and $\nu$ is a measure defined on  a sigma-algebra of subsets of $\Omega,$ we use $\|\nu\|$ for the total variation norm of $\nu.$ We will need this only when $\Omega$ is a finite set and the sigma-algebra in question is all subsets of $X,$ in which case this norm is just the $\ell^{1}$-norm of $\nu$ (with respect to the counting measure on $\Omega$). We recall some basic notions related to measurable observables.

\begin{defn}
Let $\alpha\colon X\to A,$ $\beta\colon X\to B$ be two measurable observables. Define $\alpha\vee \beta\colon X\to A\times B$ by $(\alpha\vee \beta)(x)=(\alpha(x),\beta(x)).$
We say that $\beta$ \emph{refines} $\alpha,$ and write $\alpha\leq\beta,$ if there is a measurable map $\rho\colon B\to A$ so that $\rho(\beta(x))=\alpha(x)$ for almost every $x\in X.$

\end{defn}

For the rest of the paper (except the appendix), we fix a sofic  group $G$ and a sofic approximation $\sigma_{i}\colon G\to S_{d_{i}}.$ For the rest of this section, we fix a Lebesgue probability space $(X,\mathcal{X},\mu)$ with $G\actson (X,\mathcal{X},\mu)$ by measure-preserving transformations.

\begin{defn}
Suppose that $\alpha\colon X\to A$ is a finite, measurable observable. For $\delta>0,$ and a finite $F\subseteq G,$ we let $\AP(\alpha,F,\delta,\sigma_{i})$ be the set of all $\phi\colon \{1,\dots,d_{i}\}\to A$ so that
\[\|(\phi_{\sigma_{i}}^{F})_{*}(u_{d_{i}})-(\alpha^{F})_{*}\mu\|<\delta.\]
\end{defn}
\begin{defn}
Assume that $\alpha\colon X\to A,\beta\colon X\to B,\gamma\colon X\to C$ are  finite, measurable observables with $\gamma\ge\alpha\vee \beta.$ Let $\rho_{A}\colon C\to A,\rho_{B}\colon C\to B$ be such that $\rho_{A}(\gamma(x))=\alpha(x),\rho_{B}(\gamma(x))=\beta(x)$ for almost every $x\in X.$
For a finite $F\subseteq G$ and a $\delta>0,$ set
\[\AP(\beta:\gamma,F,\delta,\sigma_{i})=\rho_{B}\circ (\AP(\gamma,F,\delta,\sigma_{i})).\]
Given $\psi\in \AP(\beta:\gamma,F,\delta,\sigma_{i}),$ set
\[\AP(\alpha|\psi:\gamma,F,\delta,\sigma_{i})=\{\rho_{A}\circ \phi:\phi \in \AP(\gamma,F,\delta,\sigma_{i}),\rho_{B}\circ \phi=\psi\}.\]
Define
\[h_{(\sigma_{i})_{i},\mu}(\alpha|\beta:\gamma,F,\delta)=\limsup_{i\to\infty}\frac{1}{d_{i}}\log\sup_{\psi\in \AP(\beta:\gamma,F,\delta,\sigma_{i})}|\AP(\alpha|\psi:\gamma,F,\delta,\sigma_{i})|.\]
\end{defn}
\begin{defn}
Let $\mathcal{F}_{1},\mathcal{F}_{2}$ be $G$-invariant sub-sigma-algebras of $\mathcal{X}.$ Suppose that $\alpha$ is a finite $\mathcal{F}_{1}$-measurable observable and that $\beta$ is a finite $\mathcal{F}_{2}$-measurable observable. Suppose that $\gamma$ is a finite observable with $\gamma\geq \alpha\vee \beta. $ Define
\[h_{(\sigma_{i})_{i},\mu}(\alpha|\beta:\gamma)=\inf_{\substack{F\subseteq G\textnormal{ finite},\\ \delta>0}}h_{(\sigma_{i})_{i},\mu}(\alpha|\beta:\gamma,F,\delta),\]
\[h_{(\sigma_{i})_{i},\mu}(\alpha|\beta:\mathcal{X})=\inf_{\gamma}h_{(\sigma_{i})_{i},\mu}(\alpha|\beta:\gamma),\]
where the infimum over all finite measurable observables $\gamma$ with $\gamma\geq \alpha\vee \beta.$  We then define
\[h_{(\sigma_{i})_{i},\mu}(\alpha|\mathcal{F}_{2}:\mathcal{X})=\inf_{\beta}h_{(\sigma_{i})_{i},\mu}(\alpha|\beta:\mathcal{X}),\]
\[h_{(\sigma_{i})_{i},\mu}(\mathcal{F}_{1}|\mathcal{F}_{2}:\mathcal{X})=\sup_{\alpha}h_{(\sigma_{i})_{i},\mu}(\alpha|\mathcal{F}_{2}:\mathcal{X}),\]
where the infimum is over all finite $\mathcal{F}_{2}$-measurable observables $\beta$ and the supremum is over all finite $\mathcal{F}_{1}$-measurable observables $\alpha.$ We call $h_{(\sigma_{i})_{i},\mu}(\mathcal{F}_{1}|\mathcal{F}_{2}:\mathcal{X})$ the relative sofic entropy of $\mathcal{F}_{1}$ given $\mathcal{F}_{2}$ in the presence of $\mathcal{X}.$
\end{defn}

We will often blur the lines between sub-sigma-algebras and factors in the notation. Thus if $Y$ is the factor corresponding to $\mathcal{F}_{2}$ we will often write $h_{(\sigma_{i})_{i},\mu}(\mathcal{F}_{1}|Y:X)$ for $h_{(\sigma_{i})_{i},\mu}(\mathcal{F}_{1}|\mathcal{F}_{2}:\mathcal{X}).$

Let us recall the usual definition of relative entropy of observables. Let
$\alpha\colon X\to A,$ $\beta\colon X\to B$ two finite measurable observables. The relative entropy of $\alpha$ given $\beta,$ denoted $H(\alpha|\beta)$ is defined by
\[H(\alpha|\beta)=-\sum_{a\in A,b\in B}\mu(\alpha^{-1}(\{a\})\cap \beta^{-1}(\{b\}))\log\left(\frac{\mu(\alpha^{-1}(\{a\})\cap \beta^{-1}(\{b\}))}{\mu(\beta^{-1}(\{b\}))}\right).\]
The above formula is an information theoretic definition of relative entropy. The following proposition is implied by Lemma 2.13 of \cite{CsizarKo} and shows that relative entropy can also be thought of as a statistical mechanics quantity: it measures how many approximations of $\alpha\vee \beta$ there are which extend any given approximation of $\beta.$

\begin{prop}\label{P:Stirling!} Let
$\alpha\colon X\to A,$ $\beta\colon X\to B$ two finite, measurable observables. Let $\rho_{A}\colon A\times B\to A,$ $\rho_{B}\colon A\times B\to B$ be the projection maps $\rho_{A}(a,b)=a,$ $\rho_{B}(a,b)=b.$ For $\delta>0$ and $n\in \N,$ let $\Xi(\alpha\vee \beta,\delta,n)$ be the set of all $\phi\in (A\times B)^{n}$ so that
\[\|(\phi)_{*}(u_{n})-(\alpha\vee \beta)_{*}\mu\|<\delta.\]
For $\psi\in B^{n},$ let
\[\Xi(\alpha|\psi,\delta,n)=\{\phi\in A^{n}:(\phi,\psi)\in \Xi(\alpha\vee \beta,\delta,n)\}.\]
Then
\[H(\alpha|\beta)=\inf_{\delta>0}\limsup_{n\to\infty}\sup_{\psi\in \rho_{B}\circ \Xi(\alpha\vee \beta,\delta,n)}\frac{1}{n}\log |\Xi(\alpha|\psi,\delta,n)|.\]

\end{prop}

We prove some easy properties of relative sofic entropy.

\begin{prop}\label{P:basicpropertieslakdhgkl}
Fix finite, measurable observables $\alpha,\beta,\gamma$ with domain $X$ and so that $\gamma\geq \alpha\vee \beta.$ We have the following properties of relative sofic entropy:
\begin{enumerate}[(i)]
\item $h_{(\sigma_{i})_{i},\mu}(\alpha|\beta:\gamma)\leq H(\alpha|\beta),$\label{I:bound1lasdfhgl}\\
\item $h_{(\sigma_{i})_{i},\mu}(\alpha|\beta:\gamma)\leq h_{(\sigma_{i})_{i},\mu}(\alpha'|\beta:\gamma) +h_{(\sigma_{i})_{i},\mu}(\alpha|\alpha':\gamma)$ if $\alpha'$ is a finite, measurable observable and $\gamma\geq \alpha'\vee \beta,$ \label{I:bound2alkghal}\\
\item $h_{(\sigma_{i})_{i},\mu}(\alpha|\beta:\gamma)\leq h_{(\sigma_{i})_{i},\mu}(\alpha|\beta':\gamma) +h_{(\sigma_{i})_{i},\mu}(\beta'|\beta:\gamma)$ if $\beta'$ is a finite, measurable observable and $\gamma\geq \alpha\vee \beta'.$ \label{I:bound3algak}\\
\item For every finite $F\subseteq G,$ every $\delta>0,$  every $i\in \N,$ and all $\psi\in\AP(\beta:\gamma,F,\delta,\sigma_{i}),$
\[|\AP(\alpha|\psi:\gamma,F,\delta,\sigma_{i})|=|\AP(\alpha\vee \beta|\psi:\gamma,F,\delta,\sigma_{i})|.\]
In particular,
\[h_{(\sigma_{i})_{i},\mu}(\alpha|\beta:\gamma,F,\delta)=h_{(\sigma_{i})_{i},\mu}(\alpha\vee \beta|\beta:\gamma,F,\delta) \textnormal{ \emph{and} } h_{(\sigma_{i})_{i},\mu}(\alpha|\beta:\gamma)=h_{(\sigma_{i})_{i},\mu}(\alpha\vee\beta|\beta:\gamma).\]
\label{I:this is obvious yo!}
\item If $\alpha_{0}$ is a measurable observable with domain $X$ and with $\alpha_{0}\leq \alpha,$ then for any $g\in G$ we have:
\[h_{(\sigma_{i})_{i},\mu}(\alpha\vee g\alpha_{0}|\beta:\mathcal{X})=h_{(\sigma_{i})_{i},\mu}(\alpha|\beta:\mathcal{X}).\]
\label{I:invariance1aldkg}
\item If $\beta_{0}$ is a measurable observable  with domain $X$ and $\beta_{0}\leq \beta,$ then for any $g\in G$ we have:
\[h_{(\sigma_{i})_{i},\mu}(\alpha|\beta:\mathcal{X})=h_{(\sigma_{i})_{i},\mu}(\alpha|\beta\vee g\beta_{0}:\mathcal{X}).\]\label{I:invariance2adlghal}

\end{enumerate}
\end{prop}

\begin{proof}
Item (\ref{I:bound1lasdfhgl}) is a direct application of Proposition \ref{P:Stirling!}. For the remaining items, let $A,A',B,C$ be the codomains of $\alpha,\alpha',\beta,\gamma$ respectively.

(\ref{I:bound2alkghal}): Let $\rho_{A}\colon C\to A,$ $\rho_{A'}\colon C\to A' ,$ and $\rho_{B}\colon C\to B$ be such that
\[\rho_{A}(\gamma(x))=\alpha(x),\,\rho_{A'}(\gamma(x))=\alpha'(x),\,\textnormal{ and } \rho_{B}(\gamma(x))=\beta(x)\]
for almost every $x\in X.$
Fix a finite $F\subseteq G$ and a $\delta>0,$ and let $\psi\in \AP(\beta:\gamma,F,\delta,\sigma_{i}).$ We have that
\[|\AP(\alpha|\psi:\gamma,F,\delta,\sigma_{i})|\leq \sum_{\phi'\in \AP(\alpha'|\psi:\gamma,F,\delta,\sigma_{i})}|\{\phi\in \AP(\alpha|\psi:\gamma,F,\delta,\sigma_{i}):(\phi,\phi')\in \AP(\alpha\vee \alpha'|\psi:\gamma,F,\delta,\sigma_{i})\}|.\]
Observe that if $(\phi,\phi')\in \AP(\alpha\vee \alpha'|\psi:\gamma,F,\delta,\sigma_{i}),$ then $(\phi,\phi')\in \AP(\alpha\vee \alpha':\gamma,F,\delta,\sigma_{i}),$ and so $\phi\in \AP(\alpha|\phi':\gamma,F,\delta,\sigma_{i}).$ We thus see that
\[|\AP(\alpha|\psi:\gamma,F,\delta,\sigma_{i})|\leq |\AP(\alpha'|\psi:\gamma,F,\delta,\sigma_{i})|\sup_{\widetilde{\phi}}|\AP(\alpha|\widetilde{\phi}:\gamma,F,\delta,\sigma_{i})|,\]
where the supremum is over all $\widetilde{\phi}\in \AP(\alpha':\gamma,F,\delta,\sigma_{i}).$ Taking the supremum over $\psi,$ applying $\frac{1}{d_{i}}\log$ to both sides, and letting $i\to\infty$ proves that
\[h_{(\sigma_{i})_{i},\mu}(\alpha|\beta:\gamma,F,\delta)\leq h_{(\sigma_{i})_{i},\mu}(\alpha'|\beta:\gamma,F,\delta,\sigma_{i})+h_{(\sigma_{i})_{i},\mu}(\alpha|\alpha':\gamma,F,\delta).\]
Now taking the infimum over $F,\delta$ completes the proof.

(\ref{I:bound3algak}): This is proved in the same way as (\ref{I:bound2alkghal}).

(\ref{I:this is obvious yo!}):
Let $\rho_{A}\colon C\to A,$ and $\rho_{B}\colon C\to B$ be such that $\rho_{A}\circ \gamma=\alpha,\rho_{B}\circ \gamma=\beta$ almost everywhere. Define $\rho_{A\times B}\colon C\to A\times B$ by $\rho_{A\times B}(c)=(\rho_{A}(c),\rho_{B}(c))$ for all $c\in C.$ Then $\rho_{A\times B}\circ \gamma=\alpha\vee \beta$ almost everywhere. Define $\pi_{A}\colon A\times B\to A,$ $ \pi_{B}\colon A\times B\to B$ by $\pi_{A}(a,b)=a,\pi_{B}(a,b)=b$ for all $(a,b)\in A\times B.$

Suppose that $\psi\in \AP(\beta:\gamma,F,\delta,\sigma_{i}),$ and let $\phi\in \AP(\alpha\vee \beta|\psi:\gamma,F,\delta,\sigma_{i}).$ Let $\widetilde{\phi}\in \AP(\gamma,F,\delta,\sigma_{i})$ be such that $\rho_{B}\circ \widetilde{\phi}=\psi,$ $\rho_{A\times B}\circ \widetilde{\phi}=\phi.$ Then,
\[\pi_{B}\circ \phi=\pi_{B}\circ \rho_{A\times B}\circ \widetilde{\phi}=\rho_{B}\circ \widetilde{\phi}=\psi.\]
In other words, the projection of $\phi$ onto the second coordinate must be $\psi.$ From this it follows that composing with $\pi_{A}$ induces a bijection $\AP(\alpha\vee \beta|\psi:\gamma,F,\delta,\sigma_{i})\to \AP(\alpha|\psi:\gamma,F,\delta,\sigma_{i})$ with inverse $\theta\mapsto \theta\vee \psi.$

(\ref{I:invariance1aldkg}): As $\alpha\leq\alpha\vee g\alpha_{0},$ we have
\[h_{(\sigma_{i})_{i},\mu}(\alpha|\beta:\mathcal{X})\leq h_{(\sigma_{i})_{i},\mu}(\alpha\vee g\alpha_{0}|\beta:\mathcal{X}).\]
By (\ref{I:bound2alkghal}), (\ref{I:this is obvious yo!}) we have:
\begin{align*}
h_{(\sigma_{i})_{i},\mu}(\alpha\vee g\alpha_{0}|\beta:\mathcal{X})&\leq h_{(\sigma_{i})_{i},\mu}(\alpha|\beta:\mathcal{X})+h_{(\sigma_{i})_{i},\mu}(g\alpha_{0}|\alpha:\mathcal{X})\\
&\leq h_{(\sigma_{i})_{i},\mu}(\alpha|\beta:\mathcal{X})+h_{(\sigma_{i})_{i},\mu}(g\alpha|\alpha:\mathcal{X}),
\end{align*}
so we only have to show that
\[h_{(\sigma_{i})_{i},\mu}(g\alpha|\alpha:\mathcal{X})\leq 0.\]

To prove this, suppose we are given finite $F\subseteq G$ with $g\in F$ and a $\delta\in (0,1/2).$ Let $\rho_{A}\colon A\times A\to A$, $\rho_{A}^{g}\colon A\times A\to A$ be the projections onto the first and second factors, respectively. Fix a $\psi\in \AP(\alpha:g\alpha\vee \alpha ,F,\delta,\sigma_{i})$. Suppose we are given a $\phi\in \AP(g\alpha \vee \alpha,F,\delta,\sigma_{i})$ and that $\rho_{A}\circ \phi=\psi.$ Set $\tau=\rho_{A}^{g}\circ \phi.$ We start by estimating the size of
\[u_{d_{i}}(\{j:\tau(j)\ne \psi(\sigma_{i}(g)^{-1}(j))\}).\]
We have that
\begin{align*}
u_{d_{i}}(\{j:\tau(j)\ne \psi(\sigma_{i}(g)^{-1}(j))\})&=u_{d_{i}}(\{j:\rho_{A}^{g}(\phi_{\sigma_{i}}^{F}(j)(e))\ne \rho_{A}(\psi_{\sigma_{i}}^{F}(j)(g))\})\\
&\leq \delta+\mu(\{x:\rho_{A}^{g}(\gamma^{F}(x)(e))\ne \rho_{A}(\gamma^{F}(x)(g))\})\\
&=\delta+\mu(\{x:\rho^{g}_{A}(\gamma(x))\ne \rho_{A}(\gamma(g^{-1}x))\})\\
&=\delta+\mu(\{x:(g\alpha)(x)\ne \alpha(g^{-1}x)\})\\
&=\delta.
\end{align*}
Thus, given $\psi,$ we can determine $\tau$ on a subset of $\{1,\dots,d_{i}\}$ of size at least $(1-\delta)d_{i}.$ The number of subsets of $\{1,\dots,d_{i}\}$ of size at most $\delta d_{i}$ is at most
\[\sum_{r=1}^{\lfloor{\delta d_{i}\rfloor}}\binom{d_{i}}{r}\leq \delta d_{i}\binom{d_{i}}{\lfloor{\delta d_{i}\rfloor}},\]
as $\delta<1/2.$ Thus
\[|\AP(g\alpha|\psi;F,\delta,\sigma_{i})|\leq \delta d_{i} \binom{d_{i}}{\lfloor{\delta d_{i}\rfloor}}|A|^{\delta d_{i}}.\]
As $\psi$ was arbitrary,
\begin{align*}
h_{(\sigma_{i})_{i},\mu}(g\alpha|\alpha:g\alpha \vee \alpha, F,\delta )&\leq \delta \log |A|+\limsup_{i\to\infty}\frac{1}{d_{i}}\log \binom{d_{i}}{\lfloor{\delta d_{i}\rfloor}}\\
&= \delta \log|A|-\delta \log(\delta)-(1-\delta)\log(1-\delta),
\end{align*}
where in the last line we apply Stirling's formula.
Letting $\delta\to 0$  and taking the infimum over $F$ shows that
\[h_{(\sigma_{i})_{i},\mu}(g\alpha|\alpha:\mathcal{X})\leq h_{(\sigma_{i})_{i},\mu}(g\alpha|\alpha:g\alpha\vee \alpha)\leq 0.\]

(\ref{I:invariance2adlghal}): It is clear that
$h_{(\sigma_{i})_{i},\mu}(\alpha|\beta\vee g\beta_{0}:\mathcal{X})\leq h_{(\sigma_{i})_{i},\mu}(\alpha|\beta:\mathcal{X}).$
 By (\ref{I:bound3algak}),(\ref{I:this is obvious yo!}) we have
\begin{align*}
h_{(\sigma_{i})_{i},\mu}(\alpha|\beta:\mathcal{X})&\leq h_{(\sigma_{i})_{i},\mu}(\alpha|\beta\vee g\beta_{0}:\mathcal{X})+h_{(\sigma_{i})_{i},\mu}(g\beta_{0}|\beta:\mathcal{X})\\
&\leq h_{(\sigma_{i})_{i},\mu}(\alpha|\beta\vee g\beta_{0}:\mathcal{X})+h_{(\sigma_{i})_{i},\mu}(g\beta|\beta:\mathcal{X}).
\end{align*}
We saw in the last step that
$h_{(\sigma_{i})_{i},\mu}(g\beta|\beta:\mathcal{X})\leq 0,$
 so this completes the proof.

\end{proof}
From Proposition \ref{P:basicpropertieslakdhgkl} one can directly show the following.

\begin{prop}\label{P:basicpropertiesagainlshdl}
Fix $G$-invariant sigma-algebras $\mathcal{F}_{1},\mathcal{F}_{2}\subseteq \mathcal{X}.$ We have the following properties of relative sofic entropy:
\begin{enumerate}[(i)]
\item $h_{(\sigma_{i})_{i},\mu}(\mathcal{F}_{2}:\mathcal{X})\leq h_{(\sigma_{i})_{i},\mu}(\mathcal{F}_{1}:\mathcal{X})+h_{(\sigma_{i})_{i},\mu}(\mathcal{F}_{2}|\mathcal{F}_{1}:\mathcal{X}),$\\ \label{P;subaddaklshdgila}
\item $h_{(\sigma_{i})_{i},\mu}(\alpha|\mathcal{F}_{1}:\mathcal{X})\leq H(\alpha|\mathcal{F}_{1})$ for any finite $\mathcal{X}$-measurable observable $\alpha,$ \\ \label{P:trivialboudnaasdlj}
\item $h_{(\sigma_{i})_{i},\mu}(\alpha|\mathcal{F}_{1}:\mathcal{X})\leq H(\alpha|\mathcal{F}_{2})+h_{(\sigma_{i})_{i},\mu}(\mathcal{F}_{2}|\mathcal{F}_{1}:\mathcal{X})$ for any finite $\mathcal{X}$-measurable observable $\alpha.$ \label{P:trivial bound 2 alksjlaj}

\end{enumerate}

\end{prop}

We can now show that relative entropy can be computed by restricting our observables to live in a generating subalgebra.

 \begin{prop}\label{P;generatohklshklas}
Let $\mathcal{S}_{1},\mathcal{S}_{2}$ be subalgebras (not necessarily sigma-algebras) of $\mathcal{X},$ and let $\mathcal{F}_{1},\mathcal{F}_{2}$ be the smallest, $G$-invariant, complete sub-sigma-algebras of $\mathcal{X}$ containing $\mathcal{S}_{1},\mathcal{S}_{2}.$ Suppose that $\mathcal{T}\supseteq \mathcal{S}_{1}\cup \mathcal{S}_{2}$ is a subalgebra of $\mathcal{X}$ so that $\mathcal{X}$ is the smallest $G$-invariant, complete sub-sigma-algebra of $\mathcal{X}$ containing $\mathcal{T}.$
 \begin{enumerate}[(i)]
 \item For any $\mathcal{T}$-measurable observables $\alpha,\beta$ we have
 \[h_{(\sigma_{i})_{i},\mu}(\alpha|\beta:\mathcal{X})=\inf_{\gamma}h_{(\sigma_{i})_{i},\mu}(\alpha|\beta:\gamma),\]
 where the infimum is over all $\mathcal{T}$-measurable observables $\gamma$ with domain $X$ and $\gamma \geq \alpha\vee \beta.$ \label{I:presenceinf}
 \item For any $\mathcal{T}$-measurable observable $\alpha$ we have
 \[h_{(\sigma_{i})_{i},\mu}(\alpha|\mathcal{F}_{2}:\mathcal{X})=\inf_{\beta}h_{(\sigma_{i})_{i},\mu}(\alpha|\beta:\mathcal{X}),\]
 where the infimum is over all $\mathcal{S}_{2}$-measurable observables $\beta$.\label{I:conditionainf}
 \item We have
 \[h_{(\sigma_{i})_{i},\mu}(\mathcal{F}_{1}|\mathcal{F}_{2}:\mathcal{X})=\sup_{\alpha}h_{(\sigma_{i})_{i},\mu}(\alpha|\mathcal{F}_{2}:\mathcal{X}),\]
 where the supremum is over all $\mathcal{S}_{1}$-measurable observables $\alpha.$\label{I:conditionalsup}
 \end{enumerate}

 \end{prop}
\begin{proof}
(\ref{I:presenceinf}): Fix a measurable observable $\gamma',$ a finite $F'\subseteq G,$ and a $\delta'>0.$ Since $\mathcal{X}$ is the smallest $G$-invariant, complete  sub-sigma-algebra containing $\mathcal{T},$ we may find a $\mathcal{T}$-measurable observable $\gamma_{0}$ with $\gamma_{0}\geq \alpha\vee \beta,$ a finite $F\subseteq G,$ and a $\delta>0$ so that for all large $i,$
\[\AP(\alpha\vee \beta:\gamma_{0},F,\delta,\sigma_{i})\subseteq \AP(\alpha\vee \beta:\gamma',F',\delta',\sigma_{i}).\]
So for every $\psi\in \AP(\beta:\gamma_{0},F,\delta,\sigma_{i})$ and  all sufficiently large $i,$
\[\AP(\alpha|\psi:\gamma_{0},F,\delta,\sigma_{i})\subseteq\AP(\alpha|\psi:\gamma',F',\delta',\sigma_{i}).\]
Thus,
\[h_{(\sigma_{i})_{i},\mu}(\alpha|\beta:\gamma_{0})\leq h_{(\sigma_{i})_{i},\mu}(\alpha|\beta:\gamma_{0},F,\delta,\sigma_{i})\leq h_{(\sigma_{i})_{i},\mu}(\alpha|\beta:\gamma',F',\delta',\sigma_{i}).\]
So
\[\inf_{\gamma}h_{(\sigma_{i})_{i},\mu}(\alpha|\beta:\gamma)\leq h_{(\sigma_{i})_{i},\mu}(\alpha|\beta:\gamma',F',\delta',\sigma_{i}),\]
where the infimum is over all finite $\mathcal{T}$-measurable observables $\gamma.$ Infimizing over $\gamma',F',\delta'$ proves $(i).$

(\ref{I:conditionainf}): It is clear that
\[\inf_{\beta}h_{(\sigma_{i})_{i},\mu}(\alpha|\beta:\mathcal{X})\geq h_{(\sigma_{i})_{i},\mu}(\alpha|\mathcal{F}_{2}:\mathcal{X})\]
where the infimum is over all $\mathcal{S}_{2}$-measurable observables $\beta.$ To prove the reverse inequality, fix an $\varepsilon>0$ and a finite $\mathcal{F}_{2}$-measurable observable $\beta'\colon X\to B'.$ Since $\mathcal{S}_{2}$ generates $\mathcal{F}_{2},$ we may find a finite  $F\subseteq G$ containing the identity, and an $\mathcal{S}_{2}$-measurable observable $\beta_{0}\colon X\to B_{0}$ such that
\[H\left(\beta'\bigg|\bigvee_{g\in F}g\beta_{0}\right)<\varepsilon.\]
By Proposition \ref{P:basicpropertieslakdhgkl} (\ref{I:bound1lasdfhgl}) and (\ref{I:bound3algak}),
\[h_{(\sigma_{i})_{i},\mu}(\alpha|\beta':\mathcal{X})\geq -\varepsilon+h_{(\sigma_{i})_{i},\mu}\left(\alpha\bigg| \bigvee_{g\in F}g\beta_{0}:\mathcal{X}\right).\]
By repeated applications of Proposition \ref{P:basicpropertieslakdhgkl}  (\ref{I:invariance2adlghal}), it follows that
\[h_{(\sigma_{i})_{i},\mu}\left(\alpha\bigg| \bigvee_{g\in F}g\beta_{0}:\mathcal{X}\right)=h_{(\sigma_{i})_{i},\mu}(\alpha|\beta_{0}:\mathcal{X}).\]
So
\[h_{(\sigma_{i})_{i},\mu}(\alpha|\beta':\mathcal{X})\geq -\varepsilon+h_{(\sigma_{i})_{i},\mu}(\alpha|\beta_{0}:\mathcal{X})\geq -\varepsilon+\inf_{\beta}h_{(\sigma_{i})_{i},\mu}(\alpha|\beta:\mathcal{X}),\]
where the infimum is over all $\mathcal{S}_{2}$-measurable observables $\beta.$ Letting $\varepsilon\to 0$ and taking the infimum over $\beta'$ completes the proof.

(\ref{I:conditionalsup}): This is proved in the same way as (\ref{I:conditionainf}) using Proposition \ref{P:basicpropertieslakdhgkl} (\ref{I:invariance1aldkg}) .

\end{proof}

\begin{prop}\label{P:basicpropertiessubalgsalkjalk}
Fix a $G$-invariant sub-sigma-algebra $\mathcal{F}\subseteq\mathcal{X}.$ The following properties of relative entropy in the presence of $\mathcal{X}$ hold:
\begin{enumerate}[(i)]
\item If $\mathcal{G}_{j},j=1,2$ are two $G$-invariant sub-sigma-algebras of $\mathcal{X},$ then
\[h_{(\sigma_{i})_{i},\mu}(\mathcal{G}_{1}\vee \mathcal{G}_{2}|\mathcal{F}:\mathcal{X})\leq h_{(\sigma_{i})_{i},\mu}(\mathcal{G}_{1}|\mathcal{F}:\mathcal{X})+h_{(\sigma_{i})_{i},\mu}(\mathcal{G}_{2}|\mathcal{F}:\mathcal{X}).\]\label{I:subaddalsdhgklahbasic}
\item If  $\mathcal{G}_{n}$ are an increasing sequence of sub-sigma-algebras of $\mathcal{X}$ and $\mathcal{G}$ is the sigma-algebra generated by their union, then
\[ h_{(\sigma_{i})_{i},\mu}(\mathcal{G}|\mathcal{F}:\mathcal{X})\leq \liminf_{n\to\infty}h_{(\sigma_{i})_{i},\mu}(\mathcal{G}_{n}|\mathcal{F}:\mathcal{X}).\]\label{P:inverselimitsalsjlkja}
\item If $\mathcal{F}',\mathcal{G}$ are $G$-invariant sub-sigma-algebras of $\mathcal{X}$ with $\mathcal{F}'\supseteq \mathcal{F},$ then
\[h_{(\sigma_{i})_{i},\mu}(\mathcal{G}|\mathcal{F}':\mathcal{X})\leq h_{(\sigma_{i})_{i},\mu}(\mathcal{G}|\mathcal{F}:\mathcal{X}).\]\label{P''monotonedecreaseaslJHlkja}
\item  If $\mathcal{G},\mathcal{G}'$ are $G$-invariant sub-sigma-algebras of $\mathcal{X}$ with $\mathcal{G}'\supseteq \mathcal{G},$ then
\[h_{(\sigma_{i})_{i},\mu}(\mathcal{G}|\mathcal{F}:\mathcal{X})\leq h_{(\sigma_{i})_{i},\mu}(\mathcal{G}'|\mathcal{F}:\mathcal{X}).\]\label{P:montoneincreaseslahdg}
\item If $\mathcal{G}$ is a $G$-invariant sub-sigma-algebra of $\mathcal{X},$ then
\[h_{(\sigma_{i})_{i},\mu}(\mathcal{G}|\mathcal{F}:\mathcal{X})=h_{(\sigma_{i})_{i},\mu}(\mathcal{G\vee \mathcal{F}}|\mathcal{F}:\mathcal{X}).\]
\label{I:again obvious}

\end{enumerate}

\end{prop}

\begin{proof}

(i). Fix finite $\mathcal{G}_{j}$-measurable observables $\alpha_{j}$ for $j=1,2.$
 Let $\beta_{j},j=1,2$ be finite $\mathcal{F}$-measurable observables and $\gamma$ a  finite $\mathcal{X}$-measurable observable so that $\gamma\geq \alpha_{1}\vee \alpha_{2}\vee \beta_{1}\vee \beta_{2}.$ Then
\begin{align*}
h_{(\sigma_{i})_{i},\mu}(\alpha|\mathcal{F}:\mathcal{X})\leq h_{(\sigma_{i})_{i},\mu}(\alpha_{1}\vee \alpha_{2}|\beta_{1}\vee \beta_{2}:\gamma)&\leq h_{(\sigma_{i})_{i},\mu}(\alpha_{1}|\beta_{1}\vee \beta_{2}:\gamma)
+h_{(\sigma_{i})_{i},\mu}(\alpha_{2}|\beta_{1}\vee\beta_{2}:\gamma)\\
&\leq h_{(\sigma_{i})_{i},\mu}(\alpha_{1}|\beta_{1}:\gamma)+h_{(\sigma_{i})_{i},\mu}(\alpha_{2}|\beta_{1}:\gamma).
\end{align*}
Infimizing over $\gamma,\beta_{1},\beta_{2},$
\[h_{(\sigma_{i})_{i},\mu}(\alpha_{1}\vee \alpha_{2}|\mathcal{F}:\mathcal{X})\leq  h_{(\sigma_{i})_{i},\mu}(\alpha_{1}|\mathcal{F}:\mathcal{X})+h_{(\sigma_{i})_{i},\mu}(\alpha_{2}|\mathcal{F}:\mathcal{X}).\]
Since  $\mathcal{G}_{1}\vee \mathcal{G}_{2}$ is generated by the set of all  observables of the form $\alpha_{1}\vee \alpha_{2}$ where $\alpha_{j},j=1,2$ are any finite, $\mathcal{G}_{j}$-measurable observables, the proof is completed by invoking Proposition \ref{P;generatohklshklas} (\ref{I:conditionalsup}).

(ii). Suppose $\alpha$ is a finite $\mathcal{G}$-measurable observable. Applying Proposition \ref{P:basicpropertiesagainlshdl} (\ref{P:trivial bound 2 alksjlaj}) we deduce that
\[h_{(\sigma_{i})_{i},\mu}(\alpha|\mathcal{F}:\mathcal{X})\leq h_{(\sigma_{i})_{i},\mu}(\mathcal{G}_{n}|\mathcal{F}:\mathcal{X})+H(\alpha|\mathcal{G}_{n}).\]
Letting $n\to\infty$ we see that
\[h_{(\sigma_{i})_{i},\mu}(\alpha|\mathcal{F}:\mathcal{X})\leq \liminf_{n\to\infty}h_{(\sigma_{i})_{i},\mu}(\mathcal{G}_{n}|\mathcal{F}:\mathcal{X}).\]
Now taking the supremum over $\alpha$ completes the proof.

(iii) and (iv). These are exercises in understanding the definitions.

(\ref{I:again obvious}): This is automatic from Proposition \ref{P:basicpropertieslakdhgkl} (\ref{I:this is obvious yo!}) and Proposition \ref{P;generatohklshklas} (\ref{I:conditionalsup}).

\end{proof}

Having established these general properties of relative entropy, we can show the existence of the relative outer Pinsker factor.

\begin{cor}\label{C:defnpinskfact}
 Fix a $G$-invariant sub-sigma-algebra $\mathcal{F}\subseteq\mathcal{X}.$ Let
\[\Pi=\{A\in \mathcal{X}:h_{(\sigma_{i})_{i},\mu}(\chi_{A}|\mathcal{F}:\mathcal{X})\leq 0\},\]
where we regard  $\chi_{A}$ as a map $X\to \{0,1\}$.
Then $\Pi$ is the unique, complete, maximal, $G$-invariant sub-sigma-algebra of $\mathcal{X}$ containing $\mathcal{F}$ with
\[h_{(\sigma_{i})_{i},\mu}(\Pi|\mathcal{F}:\mathcal{X})\leq 0.\]
\end{cor}
\begin{proof}
We first show that $\Pi$ is a $G$-invariant sub-sigma-algebra containing $\mathcal{F}$ and that $h_{(\sigma_{i})_{i},\mu}(\Pi|\mathcal{F}:\mathcal{X})\leq 0.$ Given $A\in \mathcal{X},$ let $S_{A}=\{\varnothing, A,A^{c},X\},$ and let $\mathcal{G}_{A}$ be the smallest $G$-invariant sub-sigma-algebra generated by $S_{A}.$ If $\omega\colon X\to B$ is a finite $S_{A}$-measurable observable, then $\omega$ is constant on $A,A^{c},$ and so takes on at most $2$ values. Thus $\omega\leq \chi_{A},$ and so $h_{(\sigma_{i})_{i}}(\omega|\mathcal{F}:\mathcal{X})\leq h_{(\sigma_{i})_{i}}(\alpha|\mathcal{F}:\mathcal{X}).$ Hence by Proposition \ref{P;generatohklshklas} (\ref{I:conditionalsup}),
\[h_{(\sigma_{i})_{i},\mu}(\chi_{A}|\mathcal{F}:\mathcal{X})=h_{(\sigma_{i})_{i},\mu}(\mathcal{G}_{A}|\mathcal{F}:\mathcal{X}).\]
From this, it is easy to derive that $\Pi$ is a $G$-invariant sub-sigma-algebra  containing $\mathcal{F}$ as follows. First, the fact that $\Pi$ is closed under complements and that $\Pi\supseteq \mathcal{F}$ is tautological. The fact that $\Pi$ is $G$-invariant follows from the fact that $\mathcal{G}_{A}=\mathcal{G}_{gA}$ for $g\in G.$ Lastly, suppose that $(A_{n})_{n=1}^{\infty}$ is a sequence of sets in $\Pi,$ and let
\[A=\bigcup_{n=1}^{\infty}A_{n}.\]
Observe that
\[\mathcal{G}_{A}\leq \bigvee_{n=1}^{\infty}\mathcal{G}_{A_{n}}.\]
From Proposition \ref{P:basicpropertiessubalgsalkjalk} (\ref{I:subaddalsdhgklahbasic}), (\ref{P:inverselimitsalsjlkja}), (\ref{P:montoneincreaseslahdg}) we have:
\begin{align*}
h_{(\sigma_{i})_{i},\mu}(\mathcal{G}_{A}|\mathcal{F}:\mathcal{X})&\leq \liminf_{n\to\infty}h_{(\sigma_{i})_{i},\mu}\left(\bigvee_{j=1}^{n}\mathcal{G}_{A_{j}}|\mathcal{F}:\mathcal{X}\right)\\
&\leq \liminf_{n\to\infty}\sum_{j=1}^{n}h_{(\sigma_{i})_{i},\mu}(\mathcal{G}_{A_{j}}|\mathcal{F}:\mathcal{X})\\
&\leq 0.
\end{align*}
We now show that $h_{(\sigma_{i})_{i},\mu}(\Pi|\mathcal{F}:\mathcal{X})\leq 0.$ Since $(X,\mathcal{X},\mu)$ is Lebesgue, we can find a countable collection of sets $A_{n}\in \Pi$ so that $\Pi=\bigvee_{n=1}^{\infty}\mathcal{G}_{A_{n}}.$ Using  Proposition \ref{P:basicpropertiessubalgsalkjalk}    (\ref{I:subaddalsdhgklahbasic}), (\ref{P:inverselimitsalsjlkja}), (\ref{P:montoneincreaseslahdg})   again shows that $h_{(\sigma_{i})_{i},\mu}(\Pi|\mathcal{F}:\mathcal{X})\leq 0.$

Now suppose that $\mathcal{G}$ is another $G$-invariant, sub-sigma-algebra of $\mathcal{X}$ containing $\mathcal{F}$ with $h_{(\sigma_{i})_{i},\mu}(\mathcal{G}|\mathcal{F}:\mathcal{X})\leq 0.$ For any $A\in \mathcal{G},$ we have that $\mathcal{G}_{A}\subseteq \mathcal{G}.$ Hence by Proposition \ref{P:basicpropertiessubalgsalkjalk} (\ref{P:montoneincreaseslahdg}),
\[h_{(\sigma_{i})_{i},\mu}(\mathcal{G}_{A}|\mathcal{F}:\mathcal{X})\leq h_{(\sigma_{i})_{i},\mu}(\mathcal{G}|\mathcal{F}:\mathcal{X})\leq 0.\]
So $A\in \Pi,$ and since $A$ was an arbitrary element of $\mathcal{G},$ we have that $\mathcal{G}\subseteq \Pi.$
\end{proof}

We will call $\Pi$ defined in Corollary \ref{C:defnpinskfact} the \emph{outer Pinsker sigma-algebra of $X$ relative to $\mathcal{F}$} with respect to $(\sigma_{i})_{i}.$ Typically we will drop ``with respect to $(\sigma_{i})_{i}$" if the sofic approximation is clear from the context. Note that if $\mathcal{F}=\{A\subseteq X:\mu(X\setminus A)=0 \mbox{ or } \mu(A)=0\},$ then $\Pi$ coincides with the outer Pinsker sigma-algebra, which we will denote by $\Pi_{(\sigma_{i})_{i}}(\mu)$. We denote the outer Pinsker sigma-algebra of $X$ relative to $\mathcal{F}$ by $\Pi_{(\sigma_{i})_{i}}(\mu|\mathcal{F})$. If $\Pi_{(\sigma_{i})_{i}}(\mu)$ is the algebra of sets which are null or conull, then we say that $G\cc (X,\mu)$ has \emph{completely positive measure-theoretic entropy in the presence} (with respect to $(\sigma_{i})_{i}$). The following are some of the most important properties of the outer Pinsker algebra.

\begin{prop}\label{P:Pinskermonotone}
We have the following properties of outer Pinsker factors.
 \begin{enumerate}[(i)]
  \item Suppose that $G\cc (Y,\mathcal{Y},\nu)$ is a factor of $G\cc (X,\mathcal{X},\mu),$ and identify $\mathcal{Y}\subseteq \mathcal{X}.$ Then for any complete, $G$-invariant sub-sigma-algebra $\mathcal{F}\subseteq \mathcal{Y}$ we have
\[\Pi_{(\sigma_{i})_{i}}(\nu|\mathcal{F})\subseteq \Pi_{(\sigma_{i})_{i}}(\mu|\mathcal{F}).\]\label{I:Pinksermonotoneextension}
\item If $\mathcal{F}\subseteq \mathcal{G}$ are $G$-invariant, complete sub-sigma-algebras of $\mathcal{X},$ then
\[\Pi_{(\sigma_{i})_{i}}(\mu|\mathcal{F})\subseteq \Pi_{(\sigma_{i})_{i}}(\mu|\mathcal{G}).\]\label{I:Pinksermonotoneinclusion}
\end{enumerate}

\end{prop}

\begin{proof}
(\ref{I:Pinksermonotoneextension}): Let $A\in \Pi_{(\sigma_{i})_{i}}(\nu|\mathcal{F}).$ Then
\[h_{(\sigma)_{i},\mu}(\chi_{A}|\mathcal{F}:\mathcal{X})\leq h_{(\sigma_{i})_{i},\mu}(\chi_{A}|\mathcal{F}:\mathcal{Y})\leq 0,\]
so $A\in \Pi_{(\sigma_{i})_{i}}(\mu|\mathcal{F}).$

(\ref{I:Pinksermonotoneinclusion}): This is obvious from Proposition \ref{P:basicpropertiessubalgsalkjalk} (\ref{P''monotonedecreaseaslJHlkja}).

\end{proof}

We end this section by comparing relative entropy to relative Rokhlin entropy as defined by Seward. Let $K$ be a countable, discrete group, let $(X_{0},\mathcal{X}_{0},\mu_{0})$ be a Lebesgue probability space, and $K\cc (X_{0},\mathcal{X}_{0},\mu_{0})$ a probability measure-preserving action.
If $\mathcal{F}_{2},\mathcal{F}_{1}$ are two complete, $K$-invariant, sub-sigma-algebras of $\mathcal{X}_{0},$ then Seward in \cite{SewardKrieger} (see the remarks after Question 11.1 of \cite{SewardKrieger}) defined the outer Rokhlin entropy of $\mathcal{F}_{2}$ relative to $\mathcal{F}_{1}$ by
\[h_{\textnormal{Rok}}^{K}(\mathcal{F}_{1}|\mathcal{F}_{2}:\mathcal{X}_{0 },\mu_{0})=\inf\{H(\alpha|\mathcal{F}_{2})\},\]
where the infimum is over all finite measurable observables $\alpha\colon X_{0}\to A$ such that the complete sigma-algebra generated by $\{k\alpha^{-1}(\{a\}):k\in K,a\in A\}$ and $\mathcal{F}_{2}$ contains $\mathcal{F}_{1}.$

\begin{prop}\label{P:RE upper bound}
For any complete, $G$-invariant, sub-sigma-algebras $\mathcal{F}_{1},\mathcal{F}_{2}\subseteq \mathcal{X}$ we have
\[h_{(\sigma_{i})_{i},\mu}(\mathcal{F}_{1}|\mathcal{F}_{2}:\mathcal{X})\leq h_{\textnormal{Rok}}^{G}(\mathcal{F}_{1}|\mathcal{F}_{2}:\mathcal{X},\mu).\]

\end{prop}

\begin{proof} Let $\alpha$ be a finite measurable observable, let $\mathcal{Y}$ be the smallest complete, $G$-invariant sub-sigma-algebra of $\mathcal{X}$ which makes $\alpha$ measurable, and suppose that $\mathcal{Y}\vee \mathcal{F}_{2}\supseteq \mathcal{F}_{1}.$ By Proposition \ref{P:basicpropertiessubalgsalkjalk} (\ref{P:montoneincreaseslahdg}), (\ref{I:again obvious}),
\[h_{(\sigma_{i})_{i},\mu}(\mathcal{F}_{1}|\mathcal{F}_{2}:\mathcal{X})\leq h_{(\sigma_{i})_{i},\mu}(\mathcal{Y}\vee \mathcal{F}_{2}|\mathcal{F}_{2}:\mathcal{X})= h_{(\sigma_{i})_{i},\mu}(\mathcal{Y}|\mathcal{F}_{2}:\mathcal{X}).\]
 By Proposition \ref{P;generatohklshklas} (\ref{I:conditionalsup}) and Proposition \ref{P:basicpropertiesagainlshdl} (\ref{P:trivialboudnaasdlj}), we have
\[h_{(\sigma_{i})_{i},\mu}(\mathcal{F}_{1}|\mathcal{F}_{2}:\mathcal{X})\leq h_{(\sigma_{i})_{i},\mu}(\mathcal{Y}|\mathcal{F}_{2}:\mathcal{X})=h_{(\sigma_{i})_{i},\mu}(\alpha|\mathcal{F}_{2}:\mathcal{X})\leq H(\alpha|\mathcal{F}_{2}).\]
Taking the infimum over all $\alpha$ completes the proof.

\end{proof}

We mention that it is automatic from Proposition \ref{P:RE upper bound} and Proposition \ref{P:basicpropertiesagainlshdl} (\ref{P;subaddaklshdgila}) that if $\mathcal{F}\subseteq \mathcal{X}$ is a complete, $G$-invariant sub-sigma algebra  and $G\actson (Y,\nu)$ is the factor corresponding to $\mathcal{F},$ then
\[h_{(\sigma_{i})_{i},\mu}(X,G)\leq h_{(\sigma_{i})_{i},\nu}(Y,G)+h_{\textnormal{Rok}}^{G}(\mathcal{X}|\mathcal{F}:\mathcal{X},\mu).\] 
After our paper appeared on the arXiv in preprint form, Alpeev-Seward obtained a  different proof of the above formula in \cite[Proposition 1.10]{SewardKrieger3}.

\section{Preliminaries on Local and Double Empirical Convergence}\label{S:dq}

\subsection{Preliminaries in the Topological Case}
In order to prove our product formula for outer Pinsker factors, we will need the notion of local and doubly empirical convergence defined in \cite{AustinAdd} (where it is referred to as doubly quenched convergence). We will use the reformulation given in \cite{Me12}. Throughout this subsection, we fix a compact, metrizable space $Z$ and an action $G\actons Z$ by homeomorphisms.

 Recall that if $K$ is a compact, metrizable space and $G\actons K$ by homeomorphisms, then a pseudometric $\rho$ on $K$ is \emph{dynamically generating} if for all $x,y\in K$ with $x\ne y,$ there is a $g\in G$ with $\rho(gx,gy)>0.$
If $\rho$ is as above and $n\in \N,$ we define $\rho_{2}$ on $K^{n}$ by
\[\rho_{2}(x,y)^{2}=\frac{1}{n}\sum_{j=1}^{n}\rho(x(j),y(j))^{2}.\]
Throughout this subsection, we will also fix a dynamically generating pseudometric $\rho$ on $Z.$

For a compact, metrizable space $X,$ we use $\Prob(X)$ for the space of completed Borel probability measures on $X.$ If $G\cc X$ by homeomorphisms, we let $\Prob_{G}(X)$ be the space of $G$-invariant elements of $\Prob(X).$ We begin by recalling the microstates space for measure-theoretic entropy in terms of a given compact model.

\begin{defn}\label{D:topmicrospace}
For a finite $F\subseteq G$ and a $\delta>0,$ we let $\Map(\rho,F,\delta,\sigma_{i})$ be the set of all $\phi\in Z^{d_{i}}$ so that
\[\rho_{2}(g\phi,\phi\circ \sigma_{i}(g))<\delta.\]
\end{defn}

We think of $\Map(\rho,F,\delta,\sigma_{i})$ as a space of ``topological microstates'' in the sense that $\phi\in \Map(\rho,F,\delta,\sigma_{i})$ gives a finitary model of $G\actson Z$ (i.e. it is approximately equivariant with approximate being measured in the topology on $Z$). For a finite $L\subseteq C(Z),$ a $\delta>0,$ and a $\mu\in \Prob(Z),$ set
\[U_{L,\delta}(\mu)=\bigcap_{f\in L}\left\{\nu\in \Prob(Z):\left|\int_{Z}f\,d\mu-\int_{Z}f\,d\nu\right|<\delta\right\}.\]
Note that $U_{L,\delta}(\mu)$ ranging over all $L,\delta$ gives a basis of neighborhoods of $\mu$ in the weak-$^{*}$ topology on $\Prob(Z).$

\begin{defn}\label{D:measmicrospace}
Let $\mu\in \Prob_{G}(Z).$ For finite sets $F\subseteq G,L\subseteq C(Z),$ and a $\delta>0,$ set
\[\Map_{\mu}(\rho,F,L,\delta,\sigma_{i})=\{\phi\in \Map(\rho,F,\delta,\sigma_{i}):\phi_{*}(u_{d_{i}})\in U_{L,\delta}(\mu)\}.\]
\end{defn}
We think of $\Map(\rho,F,L,\delta,\sigma_{i})$ as a space of ``measure-theoretic microstates'' for $G\actson (Z,\mu).$

\begin{defn}
Fix a $\mu\in\Prob_{G}(Z).$
We say that a sequence $\mu_{i}\in \Prob(Z^{d_{i}})$ \emph{locally and empirically converges to $\mu$}, and write $\mu_{i}\to^{le}\mu,$ if
we have:
\begin{itemize}
\item $u_{d_{i}}\left(\{j:\left|\int_{Z^{d_{i}}}f(x(j))\,d\mu_{i}(x)-\int_{Z}f\,d\mu\right|<\kappa\}\right)\to 1$ for all $f\in C(Z),\kappa>0,$ and\\
\item $\mu_{i}(\Map_{\mu}(\rho,F,L,\delta,\sigma_{i}))\to 1$ for all finite $F\subseteq G,L\subseteq C(Z),$ and $\delta>0.$\\
\end{itemize}
We say that $\mu_{i}$ \emph{locally and doubly empirically converges to $\mu$}, and write $\mu_{i}\to^{lde}\mu,$ if $\mu_{i}\otimes \mu_{i}\to^{le} \mu\otimes \mu.$
\end{defn}

In \cite{Me12} we equated this notion of local and empirical convergence to the original one defined by Austin in \cite{AustinAdd} (originally called quenched convergence). It thus follows from the results in \cite{AustinAdd} that the notion of local and doubly empirical convergence does not depend upon $\rho.$

\begin{defn}Let $G$ be a countable, discrete, sofic group with sofic approximation $\sigma_{i}\colon G\to S_{d_{i}}.$  Let $(X,\mu)$ be a Lebesgue probability space with $G\cc (X,\mu)$ by measure-preserving transformations. We say that $G\cc (X,\mu)$ is \emph{strongly sofic} (with respect to $(\sigma_{i})_{i}$) if there is some compact model $G\cc (Y,\nu)$ for $G\cc (X,\mu)$ and a sequence $\nu_{i}\in \Prob(Y^{d_{i}})$ with $\nu_{i}\to^{lde}\nu.$

\end{defn}
The results of \cite{AustinAdd} imply that local and empirical convergence is independent of the choice of topological model. Thus in the preceding definition we may replace ``some compact model'' with ``for any compact model.''

We mention that by work of \cite{Bow},\cite{Me5},\cite{GabSew} we have many examples of strongly sofic actions (these are all proved in Proposition 2.15 of \cite{Me12} so we will not repeat the proof). Recall that if $G$ is a countable, discrete group, then an algebraic action of $G$ is an action $G\cc X$ by automorphisms where $X$ is a compact, metrizable, abelian group.

\begin{example} Let $G$ be a countable, discrete, sofic group. The following actions are all strongly sofic with respect to any sofic approximation of $G:$
\begin{enumerate}
\item all Bernoulli actions (by \cite{Bow},\cite{AustinAdd}),
\item any algebraic action of the form $G\cc (X_{f},m_{X_{f}})$ where $f\in M_{m,n}(\Z(G))$ and $\lambda(f)$ has dense image (by \cite{Me5}),
\item any algebraic action of the form $G\cc (X_{f},m_{X_{f}})$ where $f\in M_{n}(\Z(G))$ and $\lambda(f)$ is injective (by \cite{Me5}),
\item any algebraic action of the form $G\cc (X,m_{X})$ where $X$ is a profinite group and the homoclinic group of $G\cc X$ is dense (by \cite{GabSew}).
\end{enumerate}
\end{example}

We prove a few permanence properties of strong soficity. For the proofs, it will be useful to phrase local and doubly empirical convergence in functional analytic terms, for which we introduce some notation.
For $f\in C(Z),$ we define $a_{g}(f)\in C(Z)$ by $a_{g}(f)(z)=f(g^{-1}z).$ If $\mu\in \Prob(Z)$ and $f\in C(Z)$ we will often use $\mu(f)$ for $\int_{Z}f\,d\mu.$ For $k\in \N$ and $1\leq j\leq k,$ we define $\iota_{j}\colon C(Z)\to C(Z^{k})$ by $\iota_{j}(f)(z)=f(z(j)).$ If $X,Y$ are compact spaces and $f\in C(X),g\in C(Y),$ we define $f\otimes g\in C(X\times Y)$ by $(f\otimes g)(x,y)=f(x)g(y).$ Lastly, for an integer $k$ we identify $(Z\times Z)^{k}$ with $Z^{k}\times Z^{k}$ in the natural way. With this notation, we can now phrase our functional analytic reformulation of local and doubly empirical convergence.

\begin{prop}\label{P:alternateformulation}
A sequence $\mu_{i}\in \Prob(Z^{d_{i}})$ locally and doubly empirically converges to $\mu\in C(Z)$ if and only if:
\begin{itemize}
\item $\frac{1}{d_{i}}\sum_{j=1}^{d_{i}}|\mu_{i}(\iota_{j}(f))-\mu(f)|^{2}\to 0,$ for all $f\in C(Z),$\\
\item $\mu_{i}\otimes \mu_{i}\left(\left|\mu(f_{1})\mu(f_{2})-\frac{1}{d_{i}}\sum_{j=1}^{d_{i}}\iota_{j}(f_{1})\otimes\iota_{j}( f_{2})\right|^{2}\right)\to 0,$ for all $f_{1},f_{2}\in C(Z),$ and\\
\item $\frac{1}{d_{i}}\sum_{j=1}^{d_{i}}\mu_{i}(|\iota_{j}(a_{g}(f))-\iota_{\sigma_{i}(g)(j)}(f)|^{2})\to 0,$ for all $f\in C(Z),g\in G.$ \\
\end{itemize}

\end{prop}

\begin{proof}
It is clear that if $\mu_{i}\to^{lde}\mu,$ then the three items in the proposition hold. Conversely, suppose that the three items in the proposition hold. Fix a compatible metric $\Delta$ on $Z.$ Define a metric $\widetilde{\Delta}$ on $Z\times Z$ by
\[\widetilde{\Delta}((x_{1},y_{1}),(x_{2},y_{2}))^{2}=\frac{\Delta(x_{1},y_{1})^{2}+\Delta(x_{2},y_{2})^{2}}{2}.\]
 The first two items imply, by density of $\Span\{f_{1}\otimes f_{2}:f_{1},f_{2}\in C(Z)\}$ in $C(Z\times Z)$, that
\[u_{d_{i}}\left(\left\{j:\left|\int_{Z^{d_{i}}\times Z^{d_{i}}}f(x(j),y(j))\,d(\mu_{i}\otimes \mu_{i})(x,y)-\int_{Z\times Z}f\,d(\mu\otimes \mu)\right|<\kappa\right\}\right)\to 1,\mbox{ for all $f\in C(Z\times Z),$}\]
\[\mu_{i}\otimes \mu_{i} \left(\left\{(x,y):\left|\frac{1}{d_{i}}\sum_{j=1}^{d_{i}}f(x(j),y(j))-\int_{Z\times Z}f\,d(\mu\otimes \mu)\right|<\kappa\right\}\right)\to 1\mbox{ for all $f\in C(Z\times Z).$}\]
Hence it suffices to show that for all $g\in G$
\[\int_{Z^{d_{i}}\times Z^{d_{i}}}\widetilde{\Delta}_{2}((g^{-1}x,g^{-1}y),(x\circ \sigma_{i}(g)^{-1},y\circ \sigma_{i}(g)^{-1}))^{2}\,d(\mu_{i}\otimes \mu_{i})(x,y)\to 0,\]
which is equivalent to
\begin{equation}\label{E:asympottics lde}
\int_{Z^{d_{i}}}\Delta_{2}(g^{-1}x,x\circ\sigma_{i}(g)^{-1})^{2}\,d\mu_{i}(x)\to 0.
\end{equation}
Following the arguments of Lemma A.1 of \cite{Me9}, we see that the third item in the hypotheses proposition implies (\ref{E:asympottics lde}).

\end{proof}

\begin{prop}\label{P:closure of strongly sofic}
The set of $\mu\in \Prob_{G}(Z)$ so that $G\actson (Z,\mu)$ is strongly sofic with respect to $(\sigma_{i})_{i}$ is weak$^{*}$-closed.
\end{prop}

\begin{proof} Suppose $(\mu^{(n)})_{n}$ is a sequence of elements of $\Prob_{G}(Z)$ so that $G\actson (Z,\mu^{(n)})$ is strongly sofic with respect to $(\sigma_{i})_{i}$ for every $n.$ Additionally, assume that there is a $\mu\in \Prob_{G}(Z)$ with $\mu^{(n)}\to^{wk^{*}}\mu$ as $n\to\infty.$  Let $L$ be a countable, dense subset of $C(Z)$ and write $L=\bigcup_{n=1}^{\infty}L_{n},$ where $L_{n}$ are finite sets. For each natural number $n,$ choose a sequence $\mu^{(n)}_{i}\in \Prob(Z^{d_{i}})$ so that $\mu^{(n)}_{i}\to^{lde}\mu^{(n)}$ as $i\to\infty.$ By Proposition \ref{P:alternateformulation}, we may choose a strictly increasing sequence of integers $i_{n}$ so that
\begin{itemize}
\item $\frac{1}{d_{i}}\sum_{j=1}^{d_{i}}\left|\mu^{(k)}_{i}(\iota_{j}(f))-\mu^{(k)}(f)\right|^{2}<2^{-n}$ for all $i\geq i_{n},1\leq k\leq n,f\in L_{n},$\\
\item  $\mu^{(k)}_{i}\otimes \mu^{(k)}_{i}\left(\left|\mu^{(k)}(f_{1})\mu^{(k)}(f_{2})-\frac{1}{d_{i}}\sum_{j=1}^{d_{i}}\iota_{j}(f_{1})\otimes \iota_{j}(f_{2})\right|^{2}\right)<2^{-n}$ for all $i\geq i_{n}$, $1\leq k\leq n$, $f\in L_{n},$ and\\
\item $\frac{1}{d_{i}}\sum_{j=1}^{d_{i}}\mu_{i}^{(k)}(|\iota_{j}(a_{g}(f))-\iota_{\sigma_{i}(g)(j)}(f)|^{2})<2^{-n}$ for all $i\geq i_{n},1\leq k\leq n,f\in L_{n}.$
\end{itemize}

Given $i\in \N,$ let $n(i)\in \N$ be defined by $i_{n(i)}\leq i<i_{n(i)+1}.$ Now define $\mu_{i}\in \Prob(Z^{d_{i}})$ by $\mu_{i}=\mu_{i}^{(n(i))}.$  Since $\mu^{(n)}\to^{wk^{*}}\mu$ as $n\to\infty$ and $n(i)\to \infty$ as $i\to\infty,$ it is a simple application of Proposition \ref{P:alternateformulation} to see that $\mu_{i}\to^{lde}\mu.$

\end{proof}

\subsection{Applications to strong soficity for actions on Lebesgue spaces}

We apply the results in the previous subsection to actions on Lebesgue spaces, even ones that are not given in terms of a topological model. We start with the following consequence of Proposition \ref{P:closure of strongly sofic}.

\begin{cor}\label{P:inverselimitstronglysofic}
Suppose that $G$ acts on an inverse system
\[\cdots (X_{n+1},\mu_{n+1})\to (X_{n},\mu_{n})\to (X_{n-1},\mu_{n-1})\to \cdots \to (X_{1},\mu_{1})\]
of Lebesgue probability spaces and that $G\actson (X_{n},\mu_{n})$ is strongly sofic with respect to $(\sigma_{i})_{i}$ for every $n.$ Then the inverse limit action of this inverse system is also strongly sofic with respect to $(\sigma_{i})_{i}.$

\end{cor}

\begin{proof} For natural numbers $n\leq m,$ let $\pi_{n,m}\colon X_{m}\to X_{n}$ be the equivariant connecting maps. By choosing appropriate topological models, we may assume that:
\begin{itemize}
\item for each $n,$ the set $X_{n}$ is a compact, metrizable space and $\mu_{n}$ is a completed Borel probability measure on $X_{n},$
\item for each $n,$ the action $G\actson X_{n}$ is by homeomorphisms,
\item  for each $m\leq n,$ the map $\pi_{n,m}$ is continuous,
\item for every $n,$ there is a point $x_{n}^{*}\in X_{n}$ which is fixed by the action of $G.$
\end{itemize}

Let
\[Z=\prod_{n=1}^{\infty}X_{n},\]
and let $\widetilde{\mu}_{n}$ be the unique Borel probability measure on $Z$  which satisfies
\[\int f\,d\widetilde{\mu}_{n}(x)=\int f(\pi_{1,n}(x),\pi_{2,n}(x),\cdots, x,x_{n+1}^{*},x_{n+2}^{*},\cdots)\,d\mu_{n}(x)\mbox{ for every $f\in C(\widetilde{X})$}.\]
It is easy to check that:
\begin{itemize}
\item for each $n,$ $G\actson (Z,\widetilde{\mu}_{n})$ is isomorphic, as a measure-preserving system, to $G\actson (X_{n},\mu_{n}),$
\item the measures $\widetilde{\mu}_{n}$ converge in the weak$^{*}$-topology to a measure $\widetilde{\mu},$
\item $G\actson (Z,\widetilde{\mu})$ is isomorphic, as a measure-preserving system, to the inverse limit of $G\actson (X_{n},\mu_{n}).$
\end{itemize}

From the first item we see that $G\actson (Z,\widetilde{\mu}_{n})$ is strongly sofic with respect to $(\sigma_{i})_{i}.$ By the above three items and Proposition \ref{P:closure of strongly sofic} we see that the inverse limit of $G\actson (Z,\mu_{n})$ is strongly sofic.

\end{proof}

\begin{cor}
Suppose that $(X,\mu)$ is a Lebesgue probability space and that $G\actson (X,\mu)$ by measure-preserving transformations.
If $G\actson (X,\mu)$ is strongly sofic with respect to $(\sigma_{i})_{i},$ then so is any other measure-preserving action on a Lebesgue probability space which is weakly contained in $G\actson (X,\mu).$

\end{cor}

\begin{proof}
We may assume that $X$ is a compact, metrizable space, that $G\actson X$ by homeomorphisms, and that $\mu\in \Prob_{G}(X).$ Let $G\actson (Y,\nu)$ be a measure-preserving action weakly contained in $G\actson (X,\mu),$ and with $(Y,\nu)$ a Lebesgue probability space. Choosing an appropriate compact model, we may assume that:
\begin{itemize}
\item $Y=\{0,1\}^{\N\times G},$
\item $G\actson Y$ by $(gy)(n,h)=y(n,g^{-1}h)$ for all $y\in Y, n\in \N,g,h\in G,$
\end{itemize}
For $(s,j,g)\in \{0,1\}\times \N\times G,$ let $A_{j,g}(s)=\{y\in Y:y(j,g)=s\}.$
Since $G\actson (Y,\nu)$ is weakly contained in $(X,\mu),$ for each $(s,j)\in \{0,1\}\times \N$ we may find a sequence $B^{(n)}_{j}(s)$ of Borel subsets of $X,$ so that
\begin{itemize}
\item $B^{(n)}_{j}(1)^{c}=B^{(n)}_{j}(0)$ for all $(j,n)\in \N\times \N,$
\item $\mu\left(\bigcap_{l=1}^{k}g_{l}B_{j_{l}}^{(n)}(s_{l})\right)\to_{n\to\infty}\nu\left(\bigcap_{l=1}^{k}A_{j_{l},g_{l}}(s_{l})\right)$ for all $k\in\N,$ all $g_{1},\dots,g_{k}\in G,$ all $j_{1},\dots,j_{k}\in \N,$ and all $s_{1},\dots,s_{k}\in\{0,1\}.$
\end{itemize}
Define a sequence of Borel maps $\phi^{(n)}\colon X\to Y$ by
\[\phi^{(n)}(x)(j,g)=1_{B^{(n)}_{j}(1)}(g^{-1}x).\]

Set $\nu^{(n)}=(\phi^{(n)})_{*}(\mu).$ It is direct to show, from our choice of $B_{n},$ that for all   $k\in\N,$ all $g_{1},\dots,g_{k}\in G,$ all $j_{1},\dots,j_{k}\in \N,$ and all $s_{1},\dots,s_{k}\in\{0,1\}$ we have
\[\nu^{(n)}\left(\bigcap_{l=1}^{k}A_{j_{l},g_{l}}(s_{l})\right)\to \nu\left(\bigcap_{l=1}^{k}A_{j_{l},g_{l}}(s_{l})\right).\]
By the Stone-Weierstrass theorem, the set
\[\bigcup_{k=1}^{\infty}\left\{1_{\bigcap_{l=1}^{k}A_{j_{l},g_{l}}(s_{l})}:g_{1},\dots,g_{k}\in G,s_{1},\dots,s_{k}\in\{0,1\},j_{1},\dots,j_{k}\in\N\right\}\]
has dense linear span in $C(X),$ so  $\nu^{(n)}\to_{n\to\infty} \nu$ in the weak$^{*}$ topology.
By construction, the action $G\actson (Y,\nu^{(n)})$ is a factor of $G\actson (X,\mu)$ and so is strongly sofic with respect to $(\sigma_{i})_{i}.$ We now apply Proposition \ref{P:closure of strongly sofic} with $Y=Z$ to complete the proof.

\end{proof}

 If $X,Y,A$ are sets and $\beta\colon Y\to A,$ we define $\beta\otimes 1\colon Y\times X\to A$ by $(\beta\otimes 1)(y,x)=\beta(y).$ For maps $\alpha\colon X\to A,$ $\beta\colon Y\to B$ we define $\beta\otimes \alpha\colon Y\times X\to B\times A$ by $(\beta\otimes \alpha)(y,x)=(\beta(y),\alpha(x)).$
If $(X,\mathcal{X},\mu),(Y,\mathcal{Y},\nu)$ are Lebesgue probability spaces, and $\mathcal{F}\subseteq \mathcal{Y},$ we identify $\mathcal{F}$ with the sub-sigma-algebra of $\mathcal{Y}\otimes\mathcal{X}$ which is the completion of $\{A\times X: A\in \mathcal{F}\}.$  Local and doubly empirical convergence has the following implication for relative entropy, which will be one of the crucial facts used in our proof of a Pinsker product formula.

\begin{prop}\label{P:wcindeplift}
Let $(X,\mathcal{X},\mu),$ $(Y,\mathcal{Y},\nu)$ be  Lebesgue probability spaces and $G\cc (X,\mathcal{X},\mu),$ $G\actson (Y,\mathcal{Y},\nu)$  measure-preserving actions. Suppose that $G\cc (X,\mathcal{X},\mu)$ is strongly sofic with respect to $(\sigma_{i})_{i}.$ Then for any $G$-invariant, complete, sub-sigma-algebras $\mathcal{F}_{1},\mathcal{F}_{2}\subseteq \mathcal{Y}$ we have
\[h_{(\sigma_{i})_{i},\nu}(\mathcal{F}_{1}|\mathcal{F}_{2}:\mathcal{Y})=h_{(\sigma_{i})_{i},\nu\otimes \mu}(\mathcal{F}_{1}|\mathcal{F}_{2}\otimes \mathcal{X}:\mathcal{Y}\otimes \mathcal{X})=h_{(\sigma_{i})_{i},\nu\otimes \mu}(\mathcal{F}_{1}\otimes \mathcal{X}|\mathcal{F}_{2}\otimes \mathcal{X}:\mathcal{Y}\otimes \mathcal{X}).\]

\end{prop}

\begin{proof}
It follows from Proposition \ref{P:basicpropertiessubalgsalkjalk} (\ref{P:montoneincreaseslahdg}) and (\ref{P''monotonedecreaseaslJHlkja}) that
\[h_{(\sigma_{i})_{i},\nu}(\mathcal{F}_{1}|\mathcal{F}_{2}:\mathcal{Y})\geq h_{(\sigma)_{i},\nu\otimes \mu}(\mathcal{F}_{1}|\mathcal{F}_{2}\otimes \mathcal{X}:\mathcal{Y}\otimes \mathcal{X}),\]
\[ h_{(\sigma)_{i},\nu\otimes \mu}(\mathcal{F}_{1}|\mathcal{F}_{2}\otimes \mathcal{X}:\mathcal{Y}\otimes \mathcal{X})\leq h_{(\sigma_{i})_{i},\nu\otimes \mu}(\mathcal{F}_{1}\otimes \mathcal{X}|\mathcal{F}_{2}\otimes \mathcal{X}:\mathcal{Y}\otimes \mathcal{X}).\]
Let $\mathcal{N}$ be the algebra of null sets in $Y.$ By Proposition \ref{P:basicpropertiessubalgsalkjalk} (\ref{I:subaddalsdhgklahbasic}),
\begin{align*}
h_{(\sigma_{i})_{i},\nu\otimes \mu}(\mathcal{F}_{1}\otimes \mathcal{X}|\mathcal{F}_{2}\otimes \mathcal{X}:\mathcal{Y}\otimes \mathcal{X})&\leq h_{(\sigma)_{i},\nu\otimes \mu}(\mathcal{F}_{1}|\mathcal{F}_{2}\otimes \mathcal{X}:\mathcal{Y}\otimes \mathcal{X})+h_{(\sigma_{i})_{i},\nu\otimes \mu}(\mathcal{N}\otimes \mathcal{X}|\mathcal{F}_{2}\otimes \mathcal{X}:\mathcal{Y}\otimes \mathcal{X})\\
&\leq h_{(\sigma)_{i},\nu\otimes \mu}(\mathcal{F}_{1}|\mathcal{F}_{2}\otimes \mathcal{X}:\mathcal{Y}\otimes \mathcal{X}).
\end{align*}
So we only have to prove that
\[h_{(\sigma_{i})_{i},\nu}(\mathcal{F}_{1}|\mathcal{F}_{2}:\mathcal{Y})\leq h_{(\sigma_{i})_{i},\nu\otimes \mu}(\mathcal{F}_{1}|\mathcal{F}_{2}\otimes \mathcal{X}:\mathcal{Y}\otimes \mathcal{X}).\]

For $j=1,2,$ fix finite $\mathcal{F}_{j}$-measurable observables $\alpha_{j}\colon Y\to A_{j}.$ Fix a $\mathcal{Y}$-measurable observable $\alpha\colon Y\to A$ with $\alpha\geq \alpha_{1}\vee \alpha_{2}.$
For $j=1,2$ let $\rho_{A_{j}}\colon A\to A_{j}$ be such that $\rho_{A_{j}}\circ \alpha=\alpha_{j}$ almost everywhere.
Let $\beta\colon X\to B,$ $\gamma\colon X\to C$ be measurable observables with $\beta\leq \gamma.$ Let $q\colon C\to B$ be such that $q\circ \gamma=\beta$ almost everywhere. Define $\Delta_{C}\colon X\to C^{G}$ by
\[\Delta_{C}(x)(g)=\gamma(g^{-1}x),\]
and let $\eta=(\Delta_{C})_{*}\mu.$ Since $G\cc (X,\mathcal{X},\mu)$ is strongly sofic with respect to $(\sigma_{i})_{i},$ so is the factor $(C^{G},\eta).$ So we may find a sequence $\eta_{i}\in \Prob(C^{d_{i}})$ with $\eta_{i}\to^{lde}\eta.$

	Fix a finite $F'\subseteq G$ and a $\delta'>0.$ Since $\eta_{i}\to^{lde}\eta$ we may find, by Theorem A of \cite{AustinAdd}, a finite $F\subseteq G,$ a $\delta>0,$ and an $I\in \N$ so that
	\[\eta_{i}(\{\psi:\phi\vee\psi\in \AP(\alpha\otimes \gamma,F',\delta',\sigma_{i})\})\geq \frac{1}{2}\]
 for $i\geq I$ and all $\phi\in \AP(\alpha,F,\delta,\sigma_{i})$.
Now fix an $i\geq I$ and a $\phi_{2}\in \AP(\alpha_{2}:\alpha,F,\delta,\sigma_{i}).$ Choose a
\[S\subseteq \{\phi\in \AP(\alpha,F,\delta,\sigma_{i}):\rho_{A_{2}}\circ \phi=\phi_{2}\}\]
with $|S|=|\AP(\alpha_{1}|\phi_{2}:\alpha,F,\delta,\sigma_{i})|$ and $\{\rho_{A_{1}}\circ \phi:\phi\in S\}=\AP(\alpha_{1}|\phi_{2}:\alpha,F,\delta,\sigma_{i}).$ From our choice of $F,\delta$ we have  that
\[u_{S}\otimes \eta_{i}(\{(\phi,\psi):\phi\vee \psi\in \AP(\alpha\otimes \gamma,F',\delta',\sigma_{i})\})\geq \frac{1}{2},\]
by Fubini's theorem. Applying Fubini's theorem again, we may find a $\psi\in \AP(\gamma,F',\delta',\sigma_{i})$ so that
\[|\{\phi\in S:\phi\vee \psi\in \AP(\alpha\otimes \gamma,F',\delta',\sigma_{i})\}|\geq \frac{1}{2}|S|=\frac{1}{2}|\AP(\alpha_{1}|\phi_{2}:\alpha,F,\delta,\sigma_{i})|.\]
Set $\psi_{B}=q\circ \psi.$ By construction, for all $i\geq I$
\[|\AP(\alpha_{1}\otimes 1|\phi_{2}\vee \psi_{B}:\alpha\otimes \gamma,F',\delta',\sigma_{i})|\geq \frac{1}{2}|\AP(\alpha_{1}|\phi_{2}:\alpha,F,\delta,\sigma_{i})|,\]
and so
\[\sup_{\phi_{2}\in \AP(\alpha_{2}:\alpha,F,\delta,\sigma_{i})} \frac{1}{2}|\AP(\alpha_{1}|\phi_{2}:\alpha,F,\delta,\sigma_{i})|\leq \sup_{\omega\in \AP(\alpha_{2}\otimes \beta:\alpha\otimes \gamma,F',\delta',\sigma_{i})}|\AP(\alpha_{1}\otimes 1|\omega:\alpha\otimes \gamma,F',\delta',\sigma_{i})|.\]
Thus
\[h_{(\sigma_{i})_{i},\nu}(\alpha_{1}|\alpha_{2}:\alpha,F,\delta)\leq h_{(\sigma_{i})_{i},\nu\otimes \mu}(\alpha_{1}\otimes 1|\alpha_{2}\otimes \beta:\alpha\otimes \gamma,F',\delta',\sigma_{i}).\]
A fortiori,
\[h_{(\sigma_{i})_{i},\nu}(\alpha_{1}|\mathcal{F}_{2}:\mathcal{Y})\leq  h_{(\sigma_{i})_{i},\nu\otimes \mu}(\alpha_{1}\otimes 1|\alpha_{2}\otimes \beta:\alpha\otimes \gamma,F',\delta',\sigma_{i}).\]
Taking the infimum over all $\beta,\alpha_{2},\alpha,\gamma,F',\delta'$ and applying Proposition \ref{P;generatohklshklas}  (\ref{I:conditionainf}) shows that
\[h_{(\sigma_{i})_{i},\nu}(\alpha_{1}|\mathcal{F}_{2}:\mathcal{Y})\leq  h_{(\sigma_{i})_{i},\nu\otimes \mu}(\alpha_{1}\otimes 1|\mathcal{F}_{2}\otimes \mathcal{X}:\mathcal{Y}\otimes \mathcal{X}).\]
Taking the supremum over all $\alpha_{1}$ completes the proof.

\end{proof}

\section{A Product Formula For Outer Pinsker Factors}\label{S:PinskerProd}

\subsection{Preliminary results on relative entropy}\label{S:PinskerProdhalf}

 Throughout this subsection, we fix a Lebesgue space $(X,\mathcal{X},\mu)$ and an action $G\actons (X,\mathcal{X},\mu)$ by measure-preserving transformations.

In this section, we prove a product formula for outer Pinsker factors of strongly sofic actions. We shall mostly follow the methods in \cite{GlasThoWe} by analyzing actions with large automorphism group.  By the automorphism group of an action we mean the following: if $H$ is a countable discrete group, if $(X_{0},\mu_{0})$ is a Lebesgue probability space, and $H\cc (X_{0},\mu_{0})$ is a probability measure preserving action, then we let $\Aut(H\cc (X_{0},\mu_{0}))$ be all measure-preserving transformations $\phi\colon X_{0}\to X_{0}$ so that $\phi(hx)=h\phi(x)$ for almost every $x\in X_{0}$ and all $h\in H.$ The main new technique in this section is   Lemma \ref{L:ultraproductrick}. It turns out that this lemma allows us to remove the freeness assumption present in \cite{GlasThoWe} without passing to orbit equivalence as in \cite{Danilenko}.

We need some notation for operators. If $\mathcal{H}$ is a Hilbert space, we use $B(\mathcal{H})$ for the algebra of bounded, linear operators on $\mathcal{H}.$  If $(X_{0},\mu_{0})$ is a Lebesgue probability space and $f\in L^{\infty}(X_{0},\mu_{0}),$ we define $m_{f}\in B(L^{2}(X_{0},\mu_{0}))$ by $(m_{f}\xi)(x)=f(x)\xi(x).$ If $(Y,\nu)$ is another Lebesgue probability space and $T\in B(L^{2}(Y,\nu)),$ we define $1\otimes T\in B(L^{2}(X_{0}\times Y,\mu_{0}\otimes \nu))$ by
\[(1\otimes T)(f)(x)=T(f(x)),\]
where in the above formula we are using the identification $L^{2}(X_{0}\times Y,\mu\otimes \nu)\cong L^{2}(X_{0},\mu,L^{2}(Y,\nu)).$
We also use (albeit very briefly) the strong operator topology on $B(L^{2}(X_{0}\times Y,\mu\otimes \nu).$ The  strong operator topology on $B(L^{2}(X_{0}\times Y,\mu\otimes \nu)$ is simply the topology of pointwise convergence in norm (i.e. the product topology) on $B(L^{2}(X_{0}\times Y,\mu\otimes \nu)).$ We often abbreviate strong operator topology by SOT.
The following argument follows closely that of Lemma 3.2 in \cite{GlasThoWe}.

\begin{lem}\label{L:cpeextension}
Let $\mathcal{F}\subseteq \mathcal{X}$ be a complete, $G$-invariant sub-sigma-algebra of $\mathcal{X}.$ Suppose that $G\cc (Y,\mathcal{Y},\nu)$ is a probability measure-preserving action on a Lebesgue probability space. Assume that $G\cc (Y,\mathcal{Y},\nu)$ is strongly sofic with respect to $(\sigma_{i})_{i}$. If $\Aut(G\cc (Y,\mathcal{Y},\nu))$ acts ergodically on $(Y,\mathcal{Y},\nu),$ then
\[\Pi_{(\sigma_{i})_{i}}(\mu\otimes \nu|\mathcal{F}\otimes \mathcal{Y})=\Pi_{(\sigma_{i})_{i}}(\mu|\mathcal{F})\otimes \mathcal{Y}.\]

\end{lem}

\begin{proof}
Let $\mathcal{Z}=\Pi_{(\sigma_{i})_{i}}(\mu|\mathcal{F})$ and $\mathcal{U}=\Pi_{(\sigma_{i})_{i}}(\mu\otimes \nu|\mathcal{F}\otimes \mathcal{Y}).$ We claim that $\mathcal{U}=\widehat{\mathcal{U}}\otimes \mathcal{Y}$ for some sigma-algebra $\widehat{\mathcal{U}}\subseteq \mathcal{X}.$ If we grant this claim, then by Proposition \ref{P:wcindeplift}
 \[h_{(\sigma_{i})_{i},\mu}(\widehat{\mathcal{U}}|\mathcal{Z}:\mathcal{X})=h_{(\sigma_{i})_{i},\mu\otimes \nu}(\widehat{\mathcal{U}}\otimes \mathcal{Y}|\mathcal{Z}\otimes \mathcal{Y}:\mathcal{X}\otimes \mathcal{Y})\leq 0.\]
By definition of $\Pi_{(\sigma_{i})_{i}}(\mu|\mathcal{F}),$ this implies that $\widehat{\mathcal{U}}\subseteq \mathcal{Z}$ and this proves the lemma. So it is enough to show that $\mathcal{U}=\widehat{\mathcal{U}}\otimes \mathcal{Y}.$

	For this, consider the conditional expectation $\E_{\mathcal{U}}$ as a projection operator
	\[\E_{\mathcal{U}}\colon L^{2}(X\times Y,\mu\otimes \nu)\to L^{2}(X\times Y,\mu\otimes \nu).\]
 For $\alpha\in \Aut(G\cc (Y,\nu)),$ define a unitary operator $U_{\alpha}\colon L^{2}(Y,\nu)\to L^{2}(Y,\nu)$ by $(U_{\alpha}\xi)=\xi\circ \alpha^{-1}.$  Recall that if $E\subseteq B(\mathcal{H})$ for some Hilbert space $\mathcal{H},$ then
\[E'=\{T\in B(\mathcal{H}):TS=ST\mbox{ for all $S\in E$}\}.\]
We first show the following:

\emph{Claim 1: if $T\in B(L^{2}(Y,\nu))$ and $T\in (\{m_{f}:f\in L^{\infty}(Y,\nu)\}\cup\{U_{\alpha}:\alpha\in \Aut(G\cc (Y,\nu))\})',$ then $T\in \C 1.$}

 To see this, first note that since $T\in \{m_{f}:f\in L^{\infty}(Y,\nu)\}',$ we have that $T=m_{k}$ for some $k\in L^{\infty}(Y,\nu)$ by Proposition 12.4 of \cite{ConwayOT}. As
$U_{\alpha}^{*}m_{k}U_{\alpha}=m_{k\circ \alpha^{-1}}$ for any $\alpha\in \Aut(G\cc (Y,\nu)),$ we see that $k=k\circ \alpha^{-1}$ for all $\alpha\in \Aut(G\cc (Y,\nu)).$ By ergodicity of the action of $\Aut(G\cc (Y,\nu)),$ we see that $T\in \C 1.$  This shows Claim 1.
	
	Note that
 \[\E_{\mathcal{U}}\in(\{1\otimes m_{f}:f\in L^{\infty}(Y,\nu)\}\cup\{1\otimes U_{\alpha}:\alpha\in \Aut(G\cc (Y,\nu))\})'.\]
 By \cite{Taka} Theorem IV.5.9, this implies that
 \begin{align*}
\E_{\mathcal{U}}&\in \overline{\Span(\{T\otimes S:T\in B(L^{2}(X,\mu)),S\in  (\{m_{f}:f\in L^{\infty}(Y,\nu)\}\cup\{U_{\alpha}:\alpha\in \Aut(G\cc (Y,\nu))\})'\})}^{SOT}\\
&=B(L^{2}(X,\mu))\otimes \C 1,
\end{align*}
 the last equality following from Claim 1. So $\E_{\mathcal{U}}=P\otimes 1$ for some projection $P\in B(L^{2}(X,\mu)).$

Let $\mathcal{\widehat{U}}=\{U\in \mathcal{X}: U\times Y\in \mathcal{U}\}.$ If we show that $P(L^{2}(X,\mathcal{X},\mu))=L^{2}(X,\widehat{\mathcal{U}},\mu),$ we will be done. It is clear that $P(L^{2}(X,\mathcal{X},\mu))\supseteq L^{2}(X,\widehat{\mathcal{U}},\mu),$ so it is enough to show that  $P(L^{2}(X,\mathcal{X},\mu))\subseteq L^{2}(X,\widehat{\mathcal{U}},\mu).$ To show this, it is enough to show that if $\xi\in P(L^{2}(X,\mu))$ and $A\subseteq \C$ is Borel, then $\xi^{-1}(A)\in \mathcal{\widehat{U}}.$ Fix such a $\xi,A.$ Since $\E_{\mathcal{U}}=P\otimes 1,$ and $P(\xi)=\xi,$ we know that $\xi\otimes 1$ is $\mathcal{U}$-measurable, because $\mathcal{U}$ is complete. So $1_{A}\circ (\xi\otimes 1)$ is $\mathcal{U}$-measurable. Since $1_{\xi^{-1}(A)\times Y}=1_{A}\circ (\xi\otimes 1),$
it follows that $\xi^{-1}(A)\times Y$ is $\mathcal{U}$-measurable, so $\xi^{-1}(A)\in \mathcal{\widehat{U}}.$

\end{proof}

We will, in fact, show that if $G\actson (Y,\nu)$ is strongly sofic with respect to $(\sigma_{i})_{i},$ then
\[\Pi_{(\sigma_{i})_{i}}(\mu\otimes \nu|\mathcal{F}\otimes \mathcal{Y})=\Pi_{(\sigma_{i})_{i}}(\mu|\mathcal{F})\otimes \mathcal{Y}\]
without assuming that $\Aut(G\cc (Y,\nu))$ acts ergodically. We reduce the general claim to the above lemma by showing that every strongly sofic action is a factor of an action which is both strongly sofic and  has a large automorphism group.

\begin{lem}\label{L:ultraproductrick}
Let $(Y,\nu)$ be a Lebesgue probability space with $G\cc (Y,\nu)$ by measure-preserving transformations.
Suppose that $G\cc (Y,\nu)$ is strongly sofic with respect to $(\sigma_{i})_{i}.$ Then there are a Lebesgue probability space $(Z,\zeta),$ and a measure-preserving action $G\cc (Z,\zeta)$ which factors onto $G\cc (Y,\nu)$ such that $G\actson (Z,\zeta)$ has the following properties:
\begin{itemize}
\item $G\actson (Z,\zeta)$ is strongly sofic with respect to $(\sigma_{i})_{i},$ and
\item $\Aut(G\actson (Z,\zeta))\actson(Z,\zeta)$ ergodically.
\end{itemize}

\end{lem}

\begin{proof}

	Let $(Z,\zeta)=(Y,\nu)^{\Z},$ and let $G\cc (Z,\zeta)$ diagonally.
If $\mu_{k}\in \Prob(Y^{d_{k}})$ and $\mu_{k}\to^{lde}\mu,$ then it follows from \cite[Theorem A]{AustinAdd} that $\mu_{k}^{\otimes n}\to^{lde}\mu^{\otimes n}$ for all $n\in \N.$  By Proposition \ref{P:inverselimitstronglysofic}, the action $G\cc (Z,\zeta)$ is strongly sofic with respect to $(\sigma_{i})_{i}.$ Let $\Z\cc (Z,\zeta)$ be the Bernoulli action. It is clear that $\Z\cc (Z,\zeta)$ commutes with $G\cc (Z,\zeta)$ and that $\Z\cc (Z,\zeta)$ is ergodic. Thus the action $\Aut(G\cc (Z,\zeta))\actson (Z,\zeta)$ is ergodic.

\end{proof}

We now remove the assumption that $\Aut(G\cc (Y,\nu))\actson (Y,\nu)$ is ergodic from Lemma \ref{L:cpeextension}.

\begin{thm}\label{T:halfpinskerprod}
Let $(Y,\mathcal{Y},\nu)$ be a Lebesgue probability space.
Suppose that $G\cc (Y,\mathcal{Y},\nu)$ is a strongly sofic action with respect to $(\sigma_{i})_{i}.$ If $\mathcal{F}\subseteq \mathcal{X}$ is a $G$-invariant, complete sub-sigma-algebra, then
\[\Pi_{(\sigma_{i})_{i}}(\mu\otimes \nu| \mathcal{F}\otimes \mathcal{Y})=\Pi_{(\sigma_{i})_{i}}(\mu|\mathcal{F})\otimes \mathcal{Y}.\]

\end{thm}

\begin{proof}

 By Lemma \ref{L:ultraproductrick}, we may find an extension $G\cc (Z,\mathcal{Z},\zeta)\to G\cc (Y,\mathcal{Y},\nu),$ where $(Z,\mathcal{Z},\zeta)$ is a Lebesgue probability space, so that $\Aut(G\cc (Z,\mathcal{Z},\zeta))$ acts on $ (Z,\mathcal{Z},\zeta)$  ergodically, and so that $G\cc (Z,\mathcal{Z},\zeta)$ is strongly sofic with respect to $(\sigma_{i})_{i}.$ If we  regard $\mathcal{Y}$ as a subalgebra of $\mathcal{Z},$ then by Proposition  \ref{P:Pinskermonotone}
 \[\Pi_{(\sigma_{i})_{i}}(\mu\otimes \nu|\mathcal{F}\otimes \mathcal{Y})\subseteq \Pi_{(\sigma_{i})_{i}}(\mu\otimes \zeta|\mathcal{F}\otimes \mathcal{Y})\subseteq \Pi_{(\sigma_{i})_{i}}(\mu\otimes \zeta|\mathcal{F}\otimes \mathcal{Z}).\]
By Lemma \ref{L:cpeextension},
 \[\Pi_{(\sigma_{i})_{i}}(\mu\otimes \nu|\mathcal{F}\otimes \mathcal{Y})\subseteq \Pi_{(\sigma_{i})_{i}}(\mu|\mathcal{F})\otimes \mathcal{Z}.\]
Since $\Pi_{(\sigma_{i})_{i}}(\mu\otimes \nu|\mathcal{F}\otimes \mathcal{Y})\subseteq \mathcal{X}\otimes \mathcal{Y},$ we have that
 \[\Pi_{(\sigma_{i})_{i}}(\mu\otimes \nu|\mathcal{F}\otimes \mathcal{Y})\subseteq (\Pi_{(\sigma_{i})_{i}}(\mu|\mathcal{F})\otimes \mathcal{Z})\cap (\mathcal{X}\otimes \mathcal{Y})=\Pi_{(\sigma_{i})_{i}}(\mu|\mathcal{F})\otimes \mathcal{Y}.\]
We remark that since we take completions everywhere we only have that $(\Pi_{(\sigma_{i})_{i}}(\mu|\mathcal{F})\otimes \mathcal{Z})\cap (\mathcal{X}\otimes \mathcal{Y})=\Pi_{(\sigma_{i})_{i}}(\mu|\mathcal{F})\otimes \mathcal{Y}$ because we are using product measures.
Because the inclusion $\Pi_{(\sigma_{i})_{i}}(\mu|\mathcal{F})\otimes \mathcal{Y}\subseteq \Pi_{(\sigma_{i})_{i}}(\mu\otimes \nu|\mathcal{F}\otimes \mathcal{Y})$ is trivial, we are done.

\end{proof}

\subsection{Proof of the Theorem \ref{T:PinskerProdIntro}:}

We now prove Theorem \ref{T:PinskerProdIntro}, but first we restate it in terms of sub-sigma-algebras.

\begin{cor}\label{C:PinskerProd}
Let $(X,\mathcal{X},\mu),(Y,\mathcal{Y},\nu)$ be Lebesgue spaces.
Suppose that $G\cc (X,\mathcal{X},\mu),G\cc (Y,\mathcal{Y},\nu)$ are actions which are strongly sofic with respect to $(\sigma_{i})_{i}.$ Then for any complete, $G$-invariant, sub-sigma-algebras $\mathcal{F}\subseteq \mathcal{X}$, $\mathcal{G}\subseteq \mathcal{Y},$ we have
\[\Pi_{(\sigma_{i})_{i}}(\mu\otimes \nu|\mathcal{F}\otimes \mathcal{G})=\Pi_{(\sigma_{i})_{i}}(\mu|\mathcal{F})\otimes \Pi_{(\sigma_{i})_{i}}(\nu|\mathcal{G}).\]
In particular, taking $\mathcal{F},\mathcal{G}$ to be the sigma-algebras of sets which are null or conull in $X,Y$ we have
\[\Pi_{(\sigma_{i})_{i}}(\mu\otimes \nu)=\Pi_{(\sigma_{i})_{i}}(\mu)\otimes \Pi_{(\sigma_{i})_{i}}(\nu).\]

\end{cor}

\begin{proof}
By Proposition  \ref{P:Pinskermonotone},
\[\Pi_{(\sigma_{i})_{i}}(\mu\otimes \nu|\mathcal{F}\otimes \mathcal{G})\subseteq \Pi_{(\sigma_{i})}(\mu\otimes \nu|\mathcal{F}\otimes \mathcal{Y})\cap \Pi_{(\sigma_{i})_{i}}(\mu\otimes \nu|\mathcal{X}\otimes \mathcal{G}).\]
By Theorem \ref{T:halfpinskerprod},
\[\Pi_{(\sigma_{i})}(\mu\otimes \nu|\mathcal{F}\otimes \mathcal{Y})=\Pi_{(\sigma_{i})_{i}}(\mu|\mathcal{F})\otimes \mathcal{Y}.\]
The action $G\actons (X,\mathcal{X},\mu)$ was arbitrary throughout Section \ref{S:PinskerProdhalf}, so Theorem \ref{T:halfpinskerprod} applies with the roles of $(X,\mathcal{X},\mu),(Y,\mathcal{Y},\nu)$ reversed. We thus have that:
\[\Pi_{(\sigma_{i})_{i}}(\mu\otimes \nu|\mathcal{X}\otimes \mathcal{G})=\mathcal{X}\otimes \Pi_{(\sigma_{i})_{i}}(\nu|\mathcal{G}).\]
Thus,
\[\Pi_{(\sigma_{i})_{i}}(\mu\otimes \nu|\mathcal{F}\otimes \mathcal{G})\subseteq (\Pi_{(\sigma_{i})}(\mu|\mathcal{F})\otimes \mathcal{Y})\cap (\mathcal{X}\otimes \Pi_{(\sigma_{i})_{i}}(\nu|\mathcal{G}))=\Pi_{(\sigma_{i})_{i}}(\mu|\mathcal{F})\otimes \Pi_{(\sigma_{i})_{i}}(\nu|\mathcal{G}).\]
As in the proof of Theorem \ref{T:halfpinskerprod}, we remark that the conclusion $\Pi_{(\sigma_{i})}(\mu|\mathcal{F})\otimes \mathcal{Y})\cap (\mathcal{X}\otimes \Pi_{(\sigma_{i})_{i}}(\nu|\mathcal{G}))=\Pi_{(\sigma_{i})_{i}}(\mu|\mathcal{F})\otimes \Pi_{(\sigma_{i})_{i}}(\nu|\mathcal{G})$ is only valid because we take the product measure on $\mu\otimes \nu.$
Since the inclusion
\[\Pi_{(\sigma_{i})_{i}}(\mu|\mathcal{F})\otimes \Pi_{(\sigma_{i})_{i}}(\nu|\mathcal{G})\subseteq \Pi_{(\sigma_{i})_{i}}(\mu\otimes \nu|\mathcal{F}\otimes \mathcal{G})\]
is trivial, we are done.

\end{proof}

\section{Applications to Complete Positive Entropy of Algebraic Actions}\label{S:algebraic}

\subsection{Generalities for an arbitrary algebraic action}\label{S:general urgh}

In this section we apply our product formula for outer Pinsker factors to the study of algebraic actions. Throughout this subsection, $X$ is a compact, metrizable group and $G\actson X$ by continuous automorphisms. Following arguments of \cite{Berg},\cite{ChungLi} we have the following description of the outer Pinsker factor for strongly sofic algebraic actions.
\begin{cor}\label{C:algebraic}
Suppose that $G\cc (X,m_{X})$ is strongly sofic with respect to $(\sigma_{i})_{i}$. Then there exists a closed, normal, $G$-invariant subgroup $Y$ of $X$ so that
$\Pi_{(\sigma_{i})_{i}}(m_{X})$ is the completion of $\{q_{Y}^{-1}(A):A\subseteq X/Y \mbox{ is $m_{X/Y}$-measurable}\},$
where $q_{Y}\colon X\to X/Y$ is the quotient map.

\end{cor}

\begin{proof}
By \cite{LindTranslation} (see also \cite{Schmidt} Lemma 20.4), it is enough to show that $x^{-1}\Pi_{(\sigma_{i})_{i}}(m_{X})\subseteq \Pi_{(\sigma_{i})_{i}}(m_{X})$ for all $x\in X$ and that $A^{-1}\in \Pi_{(\sigma_{i})_{i}}(m_{X})$ whenever $A\in \Pi_{(\sigma_{i})_{i}}(m_{X}).$ Let us first show that  $x^{-1}\Pi_{(\sigma_{i})_{i}}(m_{X})\subseteq \Pi_{(\sigma_{i})_{i}}(m_{X})$ for all $x\in X.$ Fix $A\in \Pi_{(\sigma_{i})_{i}}(m_{X}),$ and let
\[E=\{x\in X:x^{-1}A\in \Pi_{(\sigma_{i})_{i}}(m_{X})\}.\]
Observe that, since $\Pi_{(\sigma_{i})_{i}}(m_{X})$ is complete, we know that $E$ is closed in $X,$ so it is enough to show that $E$ is dense in $X.$
Let $p\colon X\times X\to X$ be the multiplication map $p(x,y)=xy.$ Then by Proposition \ref{P:Pinskermonotone} and Corollary \ref{C:PinskerProd},
\[p^{-1}(A)\in \Pi_{(\sigma_{i})_{i}}(m_{X}\otimes m_{X})=\Pi_{(\sigma_{i})_{i}}(m_{X})\otimes \Pi_{(\sigma_{i})_{i}}(m_{X}).\]
 By Fubini's Theorem, we have that
$\chi_{x^{-1}A}=\chi_{p^{-1}(A)}(x,\cdot)$ is $\Pi_{(\sigma_{i})_{i}}(m_{X})$-measurable for almost every $x\in X$. Thus $x^{-1}A\in \Pi_{(\sigma_{i})_{i}}(m_{X})$ for almost every $x\in X,$ so $E$
has full measure in $X$ and is thus dense. A similar argument shows that $A^{-1}\in \Pi_{(\sigma_{i})_{i}}(m_{X})$ (by considering the map $x\to x^{-1}$).

\end{proof}

For the next corollary we recall the definition of topological entropy in the presence due to Li-Liang in \cite{LiLiang2}.
We recall some terminology. Let $A$ be a set equipped with a pseudometric $\rho.$ For $n\in \N,$ we define $\rho_{2}$ on $A^{n}$ by
\[\rho_{2}(x,y)=\left(\frac{1}{n}\sum_{j=1}^{n}\rho(x(j),y(j))^{2}\right)^{1/2}.\]
 Given $\varepsilon>0,$ we say that a subset $B$ of $A$ is $\varepsilon$-separated with respect to $\rho$ if for every pair of unequal elements $b,b'$ in $B$ we have $\rho(b,b')>\varepsilon.$ We let $N_{\varepsilon}(A,\rho)$ be the largest cardinality of an $\varepsilon$-separated subset of $A$ with respect to $\rho.$

\begin{defn}[Definition 9.3 in \cite{LiLiang2}] \label{D:topentrpresence}
Suppose that $Z$ is a compact, metrizable space with $G\cc Z$ by homeomorphisms. Let $G\cc Y$ be a topological factor with factor map $\pi\colon Z\to Y$. Let $\rho_{Z},\rho_{Y}$ be dynamically generating pseudometrics on $Z,Y$ respectively. For a finite $F\subseteq G$ and a  $\delta>0,$ we set
\[\Map(Y:\rho_{Z},F,\delta,\sigma_{i})=\{\pi\circ \phi:\phi\in \Map(\rho_{Z},F,\delta,\sigma_{i})\}.\]
Set
\[h_{(\sigma_{i})_{i},\topo}(\rho_{Y}:\rho_{Z},F,\delta,\varepsilon)=\limsup_{i\to\infty}\frac{1}{d_{i}}\log N_{\varepsilon}(\Map(Y:\rho_{Z},F,\delta,\sigma_{i}),\rho_{Y,2}),\]
\[h_{(\sigma_{i})_{i},\topo}(\rho_{Y}:\rho_{Z},\varepsilon)=\inf_{\substack{\textnormal{ finite} F\subseteq G,\\ \delta>0}}h_{(\sigma_{i})_{i},\topo}(\rho_{Y}:\rho_{Z},F,\delta ,\varepsilon),\]
\[h_{(\sigma_{i})_{i},\topo}(Y:Z,G)=\sup_{\varepsilon>0}h_{(\sigma_{i})_{i},\topo}(\rho_{Y}:\rho_{Z},\varepsilon).\]
We call $h_{(\sigma_{i})_{i},\topo}(Z:Y,G)$ the \emph{topological entropy of $Y$ in the presence of $Z.$}
\end{defn}

By  \cite[Lemma 9.5]{LiLiang2},  (see also \cite[Theorem 2.10]{Me9} for a similar result) we know that $h_{(\sigma_{i})_{i}}(Y:Z,G)$ does not depend upon the pseudometrics $\rho_{Z},\rho_{Y}.$ Thus $h_{(\sigma_{i})_{i}}(Y:Z,G)$ is an isomorphism invariant of the factor map $\pi\colon Z\to Y$ (where isomorphism of factor maps is formulated in the obvious way).

If $Y=Z$ with factor map $\id,$ then this is just the topological entropy. We also recall the formulation of measure-theoretic entropy in the presence in terms of a topological model we gave in \cite{Me9}.

\begin{defn}[Definition 2.7 in \cite{Me9}]\label{D:measentpresence}
Let $G$ be a countable, discrete, sofic group with sofic approximation $\sigma_{i}\colon G\to S_{d_{i}}.$
Suppose that $Z$ is a compact, metrizable space with $G\cc Z$ by homeomorphisms and that $\mu\in \Prob_{G}(Z).$ Let $G\cc Y$ be a topological factor with factor map $\pi\colon Z\to Y$  and set $\nu=\pi_{*}\mu.$ Let $\rho_{Z},\rho_{Y}$ be dynamically generating pseudometrics on $Z,Y$ respectively. For finite $F\subseteq G,L\subseteq C(Z),$ and a $\delta>0,$ we set
\[\Map_{\mu}(Y:\rho_{Z},F,L,\delta,\sigma_{i})=\{\pi\circ \phi:\phi\in \Map_{\mu}(\rho_{Z},F,L,\delta,\sigma_{i})\}.\]
We define
\[h_{(\sigma_{i})_{i},\mu}(\rho_{Y}:\rho_{Z},F,L,\delta,\varepsilon)=\limsup_{i\to\infty}\frac{1}{d_{i}}\log N_{\varepsilon}(\Map_{\mu}(Y:\rho_{Z},F,L,\delta,\sigma_{i}),\rho_{Y,2}),\]
\[h_{(\sigma_{i})_{i},\mu}(\rho_{Y}:\rho_{Z},\varepsilon)=\inf_{\substack{\textnormal{ finite} F\subseteq G,\\ \textnormal{finite} L\subseteq C(Z),\\  \delta>0}}h_{(\sigma_{i})_{i},\mu}(\rho_{Y}:\rho_{Z},F,L,\delta,\varepsilon),\]
\[h_{(\sigma_{i})_{i},\mu}(\rho_{Y}:\rho_{Z},G)=\sup_{\varepsilon>0}h_{(\sigma_{i})_{i},\mu}(\rho_{Y}:\rho_{Z},\varepsilon).\]
\end{defn}
By Theorem 2.10 of \cite{Me9}, we know that $h_{(\sigma_{i})_{i},\mu}(\rho_{Z}:\rho_{X},G)$ agrees with the measure-theoretic entropy of $G\cc (Y,\nu)$ in the presence of $G\cc (Z,\mu).$
We showed in \cite{Me12} that if  $G\cc (X,m_{X})$ is strongly sofic, then for any closed, normal, $G$-invariant subgroup $Y\subseteq X$ we have $h_{\topo}(X/Y:X,G)=h_{(\sigma_{i})_{i},m_{X}}(X/Y:X,G).$ This and Corollary \ref{C:algebraic} will give us examples of algebraic actions with completely positive measure-theoretic entropy.

 If $Z$ is a compact, metrizable space with $G\cc Z$ by homeomorphisms, we say that $G\cc Z$ has \emph{completely positive topological entropy in the presence relative to $(\sigma_{i})_{i}$}  if for any topological factor $G\cc Y$ of $G\cc Z$ we have
\[h_{(\sigma_{i})_{i},\topo}(Y:Z,G)>0.\]
Recall that if $(Z,\mu)$ is a Lebesgue probability space and $G\cc (Z,\mu)$ is a probability measure-preserving action, then we say that $G\cc (Z,\mu)$ has \emph{completely positive measure-theoretic entropy in the presence relative to $(\sigma_{i})_{i}$}  if for every measure-theoretic factor $G\cc (Y,\nu)$ we have
\[h_{(\sigma_{i})_{i},\mu}(Y:Z,G)>0.\]


\begin{proof}[Proof of Corollary 1.4]\label{C:SAPPCPEalgtop}

The implications (\ref{I:algcpmeasentpres}) implies (\ref{I:algcptopentpres}) and (\ref{I:algcptopentpres}) implies (\ref{I:algtopentsubgrp}) are clear from the definitions.

Suppose that (\ref{I:algtopentsubgrp}) holds. By Corollary \ref{C:algebraic}, we may find a closed, normal, $G$-invariant subgroup $Y$ of $X$ so that the outer Pinsker factor of $G\cc (X,m_{X})$ is given by $G\cc (X/Y,m_{X/Y})$ (with the factor map being the natural homomorphism $X\to X/Y$). By Theorem 3.6 of \cite{Me12},
\[0=h_{(\sigma_{i})_{i},m_{X}}(X/Y:X,G)=h_{(\sigma_{i})_{i},\topo}(X/Y:X,G).\]
By (\ref{I:algtopentsubgrp}), it follows that $Y=X.$ Thus the outer Pinsker factor of $G\cc (X,m_{X})$ is trivial, and so $G\cc (X,m_{X})$ has completely positive measure-theoretic entropy in the presence.

\end{proof}

For the next corollary we use the IE-tuples developed by Kerr-Li (see \cite{KerrLi2} for the definition). The following may be regarded as an analogue of Corollary 8.4 of \cite{ChungLi}.
\begin{cor}\label{C:ietuplesecalg}
If $G\actson (X,m_{X})$ is strongly sofic with respect to $(\sigma_{i})_{i}$
and $\IE^{2}_{(\sigma_{i})_{i}}(X,G)=X^{2},$ then $G\cc (X,m_{X})$ has completely positive measure-theoretic entropy in the presence (relative to $(\sigma_{i})_{i}$).

\end{cor}

\begin{proof} It is clear from the definitions that if $\IE^{2}_{(\sigma_{i})_{i}}(X,G)=X^{2},$ then $G$ has completely positive topological entropy in the presence (relative to $(\sigma_{i})_{i}$).

\end{proof}

\subsection{Proof of Corollary \ref{C:algintroCPE}}

Combining with our previous results in \cite{Me7} we may produce a large class of algebraic actions with completely positive measure-theoretic entropy.

\begin{proof}[Proof of Corollary \ref{C:algintroCPE}]
Since $X$ was arbitrary in the previous subsection, Corollary \ref{C:ietuplesecalg} applies to $X=X_{f}.$ The Corollary now follows from Corollary \ref{C:ietuplesecalg} and Corollary 4.10 of \cite{Me7}.

\end{proof}

\appendix

\section{Agreement of Relative Entropy with the Amenable  Case}\label{A:Agree}
In this section we show that our definition of (upper) relative entropy for actions of sofic groups agrees with the usual definition when the group is amenable. Throughout the appendix, we suppose that $G$ is an amenable group, that $(X,\mathcal{X},\mu)$ is a Lebesgue probability space, and that $G\actson (X,\mathcal{X},\mu)$ is a probability measure-preserving action. We also fix a sofic approximation $\sigma_{i}\colon G\to S_{d_{i}}$ of $G.$

Suppose that $\mathcal{F}\subseteq \mathcal{X}$ is a $G$-invariant, complete, sub-sigma-algebra of $\mathcal{X}$ and that $\alpha\colon X\to A$ is a finite measurable observable. Recall that the relative dynamical entropy of $\alpha$ given $\mathcal{F}$ is defined by
\[h_{\mu}(\alpha|\mathcal{F},G)=\lim_{n\to\infty}\frac{H\left(\bigvee_{g\in F_{n}}g\alpha|\mathcal{F}\right)}{|F_{n}|},\]
where $F_{n}$ is a F\o lner sequence (this was first defined in \cite{AbramovRohklin3}). It is shown in \cite{AbramovRohklin3} that the above limit exists and is independent of the F\o lner sequence. If $\mathcal{G}$ is another complete, $G$-invariant, sub-sigma algebra of $\mathcal{X},$ then we define
\[h_{\mu}(\mathcal{G}|\mathcal{F},G)=\sup_{\alpha}h_{\mu}(\alpha|\mathcal{F},G),\]
where the supremum is over all finite, $\mathcal{G}$-measurable observables $\alpha.$ By \cite{AbramovRohklin3}, we know that if $\alpha$ is a finite measurable observable and $\mathcal{G}$ is the smallest complete, $G$-invariant sub-sigma-algebra of $\mathcal{X}$ containing $\alpha,$ then
\[h_{\mu}(\mathcal{G}|\mathcal{F},G)=h_{\mu}(\alpha|\mathcal{F},G).\]

We use the following simple combinatorial lemma whose proof is left as an exercise to the reader.
\begin{lem}\label{L:almost permutation lemma}
Let $K$ be a finite set and $\varepsilon>0.$ Then there is a $\kappa>0$  (depending only upon $|K|,\varepsilon$) with the following property. Suppose we are given:
\begin{itemize}
\item a natural number $d,$
\item $K$-tuples $(B_{k})_{k\in K},(C_{k})_{k\in K}$ of subsets of $\{1,\dots,d\}$ with $|u_{d}(B_{k})-u_{d}(C_{k})|<\kappa$ for all $k\in K$, and so that $\max(u_{d}(B_{k}\cap B_{l}),u_{d}(C_{k}\cap C_{l}))<\kappa$ for all $k,l\in K$ with $k\ne l,$ and
\item functions $p_{k}\colon B_{k}\to C_{k},$ for $k=1,\dots,n,$ with $u_{d}(p_{k}(B_{k}))\geq u_{d}(C_{k})-\kappa$ for all $k=1,\dots,n.$
\end{itemize}
Then there is a $p\in S_{d}$ so that
\[u_{d}\left(\bigcup_{k\in K}\{j\in B_{k}:p(j)\ne p_{k}(j)\}\right)\leq \varepsilon.\]

\end{lem}

The following lemma, roughly speaking, says the following: given a microstate $\phi$ of a probability measure-preserving action of an amenable group $G$, one can obtain \emph{any} other microstate for the same action by pre-composing $\phi$ with a permutation that almost commutes with the sofic approximation of $G.$

\begin{lem}\label{L:PopaconjationA}
Let $\beta\colon X\to B$ be a finite measurable observable.
For every $\varepsilon>0$ and every finite $E\subseteq G,$ there exist a finite $F\subseteq G,$ a $\delta>0,$ and an $I\in \N$ so that if $i\geq I$ and  $\phi,\psi\in \AP(\beta,F,\delta,\sigma_{i}),$ then there is a $p\in S_{d_{i}}$ with
\[u_{d_{i}}(\{j:\psi(p(j))\ne\phi(j)\})\leq \varepsilon\textnormal{ and }\max_{g\in E}u_{d_{i}}(\{j:p(\sigma_{i}(g)(j))\ne\sigma_{i}(g)p(j)\})\leq \varepsilon.\]

\end{lem}

Before jumping into the proof let us make a few comments about Lemma \ref{L:PopaconjationA} . If $G\actson (X,\mu)$ is free, then an alternate way to say Lemma \ref{L:PopaconjationA} is in terms of the full pseudogroup of the action $G\actson (X,\mu),$ and sofic approximations of the full pseudogroup. For the precise definition of a sofic approximation of the full pseudogroup see the discussion preceding Definition 2.1 of \cite{DKP}. Any pair $(\phi,\sigma)$ consisting of a sufficiently almost free almost homomorphism $\sigma\colon G\to S_{d},$ and a microstate $\phi$ for $G\actson (X,\mu)$ with respect to $\sigma$ gives rise, in a completely natural way, to an almost trace-preserving, almost homomorphism  $[[G\actson (X,\mu)]]\to [[S_{d}\actson \{1,\dots,d\}]].$ This gives a sofic approximation of the full pseudogroup.  By Connes-Feldman-Weiss \cite{CFW}, the orbit equivalence relation of every free action of amenable group is hyperfinite. The lemma is then asserting that if $\mathcal{R}$ is a hyperfinite equivalence relation, then any two sofic approximations of the full pseudogroup of $\mathcal{R}$ are approximately conjugate. This  is precisely what Proposition 1.20 of \cite{Pan} asserts. Intuitively, that any two sofic approximations of a full pseudogroup of a hyperfinite equivalence relation $\mathcal{R}$ are approximately conjugate should be obvious, since this fact is easy to establish when almost every equivalence class of $\mathcal{R}$ is finite. If the reader is familiar with the appropriate background on orbit equivalence relations, then we invite them to check that there is indeed a straightforward proof along these lines. This is the proof that \cite{Pan} gives. One can give a similar formulation of ``uniqueness up to approximate conjugacy" when $G\actson (X,\mu)$ is not free, using the transformation groupoid instead of the full pseudogroup. Under this formulation, Lemma \ref{L:PopaconjationA} is equivalent to a result of Popa (see Corollary 5.2 of \cite{PoppArg}), who proved it in the framework of ultraproducts of operator algebras (also relying on the Connes-Feldman-Weiss theorem). At the request of the referee, we have included a proof that does not use ultrafilters as in \cite{Pan},\cite{PoppArg}.

It is worth mentioning that our proof is essentially an application of the quasi-tiling machinery developed by Ornstein-Weiss in \cite{OrnWeiss}. The Connes-Feldman-Weiss result also ultimately relies on this quasi-tiling machinery. In this sense, the proof that we give below is not even different than the proofs given in \cite{Pan},\cite{PoppArg}, it is just a language translation of their proofs into an ultrafilter-free version.

\begin{proof}[Proof of Lemma \ref{L:PopaconjationA}]
Without loss of generality, we may assume that $\beta$ is a generating observable. We prove the lemma in two cases.

\emph{Case 1. When $G\actson (X,\mu)$ is free.}
We use the quasi-tiling machinery developed by Ornstein-Weiss. Let $\varepsilon>0$, then we may find an integer $L,$ measurable subsets $\widetilde{V}_{1},\dots,\widetilde{V}_{L}$ of $X$, and $(E,\varepsilon)$-invariant subsets $T_{1},\dots,T_{L}$ of $G$ so that
\begin{itemize}
\item $\{h\widetilde{V}_{r}\}_{h\in T_{r},1\leq r\leq L}$ is a disjoint family of sets, and
\item $\mu\left(\bigcup_{r=1}^{L}T_{r}\widetilde{V}_{r}\right)\geq 1-\varepsilon.$
\end{itemize}
By perturbing $\widetilde{V}_{r}$ slightly, we may assume that there is a finite $F_{0}\subseteq G$ so that $h\widetilde{V}_{r}$ is $\beta^{F_{0}}$-measurable for all $1\leq r\leq L$ and $h\in T_{r}.$ We may also assume that $F_{0}\supseteq \bigcup_{r=1}^{L}T_{r}^{-1}.$
  For $1\leq r\leq l$ and $h\in T_{r},$ find $\widetilde{B}_{r}\subseteq B^{F_{0}}$ so that $\widetilde{V}_{r}=(\beta^{F_{0}})^{-1}(\widetilde{B}_{r}).$
 Set $K=\{(h,r):h\in T_{r},1\leq r\leq L\},$ and let $\kappa>0$ be as in Lemma \ref{L:almost permutation lemma} for this $K,\varepsilon.$
Let $J_{i}\subseteq \{1,\dots,d_{i}\}$ be the set of $j$ so that
\begin{itemize}
\item for all $\alpha_{1},\alpha_{2}\in\{-1,1\},$ and every $h_{1},h_{2}\in \bigcup_{r=1}^{L}(E\cup \{e\}\cup E^{-1})(T_{r}\cup\{e\} \cup T_{r}^{-1})$ we have $\sigma_{i}(h_{1}^{\alpha_{1}}h_{2}^{\alpha_{2}})(j)=\sigma_{i}(h_{1})^{\alpha_{1}}\sigma_{i}(h_{2})^{\alpha_{2}}(j),$ and
\item for all $h_{1},h_{2}\in \bigcup_{r=1}^{L}(E\cup \{e\}\cup E^{-1})(T_{r}\cup\{e\} \cup T_{r}^{-1})$ with $h_{1}\ne h_{2},$ we have $\sigma_{i}(h_{1})(j)\ne \sigma_{i}(h_{2})(j).$
\end{itemize}
We then have that $u_{d_{i}}(J_{i})\to 1,$ so there is an $I_{0}\in \N$ so that $u_{d_{i}}(J_{i})\geq 1-\varepsilon$ for all $i\geq I_{0}.$

	We may choose a sufficiently small positive number $\delta\in (0,\kappa/2),$ a sufficiently large finite subset $F$ of $G,$  and a sufficiently large natural number $I_{1}$  so that for all $i\geq I_{1}$ and any $\phi\in \AP(\beta,F,\delta,\sigma_{i})$ we have
\begin{itemize}
\item $|u_{d_{i}}((\phi_{\sigma_{i}}^{F_{0}})^{-1}(\widetilde{B}_{r}))-\mu(\widetilde{V_{r}})|<\kappa/2\mbox{ for $r=1,\dots,L$},$
\item $u_{d_{i}}(\sigma_{i}(h)(\phi_{\sigma_{i}}^{F_{0}})^{-1}(\widetilde{B}_{r})\cap \sigma_{i}(k)(\phi_{\sigma_{i}}^{F_{0}})^{-1}(\widetilde{B}_{s}))<\kappa\mbox{ for all $(h,r),(k,s)\in K$ with $(h,r)\ne (k,s),$}$
\item $u_{d_{i}}\left(\bigcup_{r=1}^{L}\sigma_{i}(T_{r})(\phi_{\sigma_{i}}^{F_{0}})^{-1}(\widetilde{B}_{r})\right)\geq 1-2\varepsilon,$ and
\item $\delta\sum_{r=1}^{L}|T_{r}\setminus g^{-1}T_{r}|<\varepsilon$.
\end{itemize}
Set $I=\max(I_{0},I_{1}).$ Fix an $i\geq I$ and $\phi,\psi\in \AP(\beta,F,\delta,\sigma_{i}).$ For $1\leq r\leq L$, $h\in T_{r}\setminus\{e\},$ and $\widetilde{b}\in \widetilde{B}_{r},$ let
\[B_{e,r}^{\widetilde{b}}=(\phi_{\sigma_{i}}^{F_{0}})^{-1}(\{\widetilde{b}\}),C_{e,r}^{\widetilde{b}}=(\psi_{\sigma_{i}}^{F_{0}})^{-1}(\{\widetilde{b}\}),\]
\[B_{e,r}=(\phi_{\sigma_{i}}^{F_{0}})^{-1}(\widetilde{B}_{r}),C_{e,r}=(\psi_{\sigma_{i}}^{F_{0}})^{-1}(\widetilde{B}_{r}),\]
\[B_{h,r}=\sigma_{i}(h)B_{e,r},C_{h,r}=\sigma_{i}(h)C_{e,r}.\]

For every $1\leq r\leq L$ and every $\widetilde{b}\in \widetilde{B}_{r},$ choose a $p_{e,r}^{\widetilde{b}}\colon B_{e,r}^{\widetilde{b}}\to C_{e,r}^{\widetilde{b}}$ which is ``as bijective as possible." Namely, we require that
\[u_{d_{i}}(p_{e,r}^{\widetilde{b}}(B_{e,r}^{\widetilde{b}}))=\min(u_{d_{i}}(B_{e,r}^{\widetilde{b}}),u_{d_{i}}(C_{e,r}^{\widetilde{b}})).\]
Define $p_{e,r}\colon B_{e,r}\to C_{e,r}$ by saying that $p_{e,r}\big|_{B_{e,r}^{\widetilde{b}}}=p_{e,r}^{\widetilde{b}}$ for every $\widetilde{b}\in \widetilde{B}_{r}.$  Then
\begin{equation}\label{E:aldgjklajgkla equiv}
\phi_{\sigma_{i}}^{F_{0}}(j)=\psi_{\sigma_{i}}^{F_{0}}(p_{e,r}(j))\mbox{ for all $j\in B_{e,r}.$}
\end{equation}
Since $p_{e,r}(B_{e,r})=\bigcup_{\widetilde{b}\in B_{r}}p_{e,r}^{\widetilde{b}}(B_{e,r}^{\widetilde{b}}),$ we have that
\begin{align}\nonumber
u_{d_{i}}(p_{e,r}(B_{e,r}))=\sum_{\widetilde{b}\in \widetilde{B}_{r}}\min(u_{d_{i}}(B_{e,r}^{\widetilde{b}}),u_{d_{i}}(C_{e,r}^{\widetilde{b}}))&\geq \sum_{\widetilde{b}\in \widetilde{B}_{r}}u_{d_{i}}(C_{e,r}^{\widetilde{b}})-\sum_{\widetilde{b}\in \widetilde{B}_{r}}|u_{d_{i}}(B_{e,r}^{\widetilde{b}})-u_{d_{i}}(C_{e,r}^{\widetilde{b}})|\\  \nonumber
&=u_{d_{i}}(C_{e,r})-\sum_{\widetilde{b}\in\widetilde{B}_{r}}|(\phi^{F_{0}}_{\sigma_{i}})_{*}(u_{d_{i}})(\{\widetilde{b}\})-(\psi^{F_{0}}_{\sigma_{i}})_{*}(u_{d_{i}})(\{\widetilde{b}\})|\\ \nonumber
&\geq u_{d_{i}}(C_{e,r})-\|(\psi^{F_{0}}_{\sigma_{i}})_{*}(u_{d_{i}})-(\phi^{F_{0}}_{\sigma_{i}})_{*}(u_{d_{i}})\|\\ \nonumber
&\geq u_{d_{i}}(C_{e,r})-2\delta\\ \label{I: its the right size yo!}
&\geq u_{d_{i}}(C_{e,r})-\kappa.
\end{align}

For $1\leq r\leq L$ and $h\in T_{r}\setminus\{e\},$ define $p_{h,r}\colon B_{h,r}\to C_{h,r}$ by
\[p_{h,r}(\sigma_{i}(h)(j))=\sigma_{i}(h)p_{e,r}(j)\mbox{ for all $j\in B_{e,r}$.}\]

By  $(\ref{I: its the right size yo!}),$ and our choice of $F,\delta,I_{1},$  the hypotheses of Lemma \ref{L:almost permutation lemma} apply, so we can find a $p\in \Sym(d_{i})$  as in the conclusion of Lemma \ref{L:almost permutation lemma} for the family of functions $(p_{h,r})_{(h,r)\in K}.$ Let
\[J_{i}^{o}=J_{i}\cap p^{-1}(J_{i})\cap\left[\left(\bigcup_{r=1}^{L}\sigma_{i}(T_{r})(B_{e,r})\right)\setminus \left(\bigcup_{r=1}^{L}\bigcup_{h\in T_{r}}\{j\in B_{h,r}:p(j)\ne p_{h,r}(j)\}\right)\right]
.\]
By our choice of $p,$
\begin{align*}u_{d_{i}}((J_{i}^{0})^{c})&\leq 2(1-u_{d_{i}}(J_{i}))+u_{d_{i}}\left(\bigcup_{r=1}^{L}\bigcup_{h\in T_{r}}\{j\in B_{h,r}:p(j)\ne p_{h,r}(j)\}\right)+u_{d_{i}}\left(\left(\bigcup_{r=1}^{L}\sigma_{i}(T_{r})B_{e,r}\right)^{c}\right)\\
&\leq 2(1-u_{d_{i}}(J_{i}))+\varepsilon+1-u_{d_{i}}\left(\bigcup_{r=1}^{L}\sigma_{i}(T_{r})B_{e,r}\right)\\
&\leq 2(1-u_{d_{i}}(J_{i}))+3\varepsilon.
\end{align*}
Since $i\geq I_{0},$ we have $u_{d_{i}}((J_{i}^{0})^{c})\leq 5\varepsilon.$ So
\begin{align}\nonumber
u_{d_{i}}(\{j:\psi(p(j))\ne \phi(j)\})&\leq 5\varepsilon+\sum_{r=1}^{L}\sum_{h\in T_{r}}u_{d_{i}}(\{j\in B_{e,r}\cap J_{i}^{o}:\psi(\sigma_{i}(h)p_{e,r}(j))=\phi(\sigma_{i}(h)(j))\})\\ \nonumber
&\leq  5\varepsilon+\sum_{r=1}^{L}\sum_{h\in T_{r}}u_{d_{i}}(\{j\in B_{e,r}\cap J_{i}^{o}:\psi_{\sigma_{i}}^{F_{0}}(p_{e,r}(j))(h^{-1})=\phi_{\sigma_{i}}^{F_{0}}(j)(h^{-1})\})\\\label{E: conjugation estimatesljgrla}
&=5\varepsilon,
\end{align}
the last line following by (\ref{E:aldgjklajgkla equiv}).

Fix a $g\in E.$ We have to estimate $u_{d_{i}}(\{j:p(\sigma_{i}(g)(j))\ne\sigma_{i}(g)(p(j))\}).$
We have that
\begin{align*}
&u_{d_{i}}(\{j:p(\sigma_{i}(g)(j))\ne\sigma_{i}(g)(p(j))\})\leq 5\varepsilon+u_{d_{i}}(\{j\in J_{i}^{0}:p(\sigma_{i}(g)(j))\ne \sigma_{i}(g)p(j)\})\leq\\
& 5\varepsilon+ \sum_{r=1}^{L}\sum_{h\in T_{r}\cap g^{-1}T_{r}}u_{d_{i}}(\{j\in B_{e,r}\cap J_{i}^{o}:p(\sigma_{i}(g)\sigma_{i}(h)(j))=\sigma_{i}(g)p_{h,r}(\sigma_{i}(h)(j))\})\\
&+\sum_{r=1}^{L}|T_{r}\setminus g^{-1}T_{r}|u_{d_{i}}(B_{e,r})\leq \\
&5\varepsilon+\sum_{r=1}^{L}\sum_{h\in T_{r}\cap g^{-1}T_{r}}u_{d_{i}}(\{j\in B_{e,r}\cap J_{i}^{o}:p(\sigma_{i}(g)\sigma_{i}(h)(j))=\sigma_{i}(g)p_{h,r}(\sigma_{i}(h)(j))\})\\
&+\delta\sum_{r=1}^{L}|T_{r}\setminus g^{-1}T_{r}|+\sum_{r=1}^{L}|T_{r}\setminus g^{-1}T_{r}|\mu(\widetilde{V}_{r}).\\
\end{align*}
For every $1\leq r\leq L,h\in T_{r}\cap g^{-1}T_{r},$ and every $j\in J_{i},$ we have $\sigma_{i}(g)\sigma_{i}(h)(j)=\sigma_{i}(gh)(j).$ So  for all $1\leq r\leq L,$ $h\in T_{r}\cap g^{-1}T_{r},$ $j\in J_{i}^{0}\cap B_{e,r},$
\begin{align*}
p(\sigma_{i}(g)\sigma_{i}(h)(j))=p(\sigma_{i}(gh)(j))=p_{gh,r}(\sigma_{i}(gh)(j))=\sigma_{i}(gh)p_{e,r}(j)&=\sigma_{i}(g)\sigma_{i}(h)p_{e,r}(j)\\
&=\sigma_{i}(g)p_{h,r}(\sigma_{i}(h)(j)).
\end{align*}
Additionally, our choice of $\delta$ implies that $\delta\sum_{r=1}^{L}|T_{r}\setminus g^{-1}T_{r}|<\varepsilon.$ Hence,
\begin{align*}
u_{d_{i}}(\{j:p(\sigma_{i}(g)(j))\ne \sigma_{i}(g)(p(j))\})&\leq 6\varepsilon+\sum_{r=1}^{L}|T_{r}\setminus g^{-1}T_{r}|\mu(\widetilde{V}_{r})\\
&\leq 6\varepsilon+\varepsilon\sum_{r=1}^{L}|T_{r}|\mu(\widetilde{V}_{r})\\
&\leq 6\varepsilon+\varepsilon\mu\left(\bigcup_{r=1}^{L}T_{r}\widetilde{V}_{r}\right)\\
&\leq 7\varepsilon,
\end{align*}
the third to last inequality follows because $T_{r}$ is $(E,\varepsilon)$-invariant, and the second to last inequality follows because $\{gV_{r}\}_{1\leq r\leq L,g\in T_{r}}$ is a disjoint family of sets. Thus,
\[u_{d_{i}}(\{j:p(\sigma_{i}(g)(j))\ne\sigma_{i}(g)(p(j))\})\leq 7\varepsilon.\]
Since $\varepsilon$ is arbitrary, the above estimate and (\ref{E: conjugation estimatesljgrla}) complete the proof of Case 1.

\emph{Case 2. The case of general $G\actson (X,\mu)$.}
In this case, consider the diagonal action $G\actson (X\times \{0,1\}^{G},\mu\otimes u_{\{0,1\}}^{\otimes G}),$ where $G\actson (\{0,1\},u_{\{0,1\}})^{G}$ is the Bernoulli action.   Let $\widetilde{\beta}\colon X\times \{0,1\}^{G} \to B\times \{0,1\}$ be defined by $\widetilde{\beta}(x,y)=(\beta(x),y(e)).$ Since the action $G\actson (X\times \{0,1\}^{G},\mu\otimes u_{\{0,1\}}^{\otimes G})$ is free, by Case 1 we may choose a finite $\widetilde{F}\subseteq G,$ a $\widetilde{\delta}>0,$ and an $\widetilde{I}\in \N$ so that if $i\geq \widetilde{I}$ and
$\phi,\psi\in \AP(\widetilde{\beta},\widetilde{F},\widetilde{\delta},\sigma_{i}),$ then there is a $p\in S_{d_{i}}$ so that
\[\max_{g\in E} u_{d_{i}}(\{j:p(\sigma_{i}(g)(j))\ne \sigma_{i}(g)(p(j))\})\leq \varepsilon \textnormal{ and } u_{d_{i}}(\{j:\phi(p(j))\ne \psi(p(j))\})\leq \varepsilon.\]

By the proof of Theorem 8.1 of \cite{Bow}, we may choose a finite $F\subseteq G,$ a $\delta>0,$ and an $I\in \N,$  so that if $i\geq I$ and $\psi\in \AP(\beta,F,\delta,\sigma_{i}),$ then there is a
$\widetilde{\psi}\in \AP(\widetilde{\beta},\widetilde{F},\widetilde{\delta},\sigma_{i})$ with $\rho_{A}\circ \widetilde{\psi}=\psi.$
Suppose that $i\geq I$ and that $\psi,\phi\in \AP(\beta,F,\delta,\sigma_{i}).$ Choose $\widetilde{\psi},\widetilde{\phi}\in \AP(\widetilde{\beta},\widetilde{F},\widetilde{\delta},\sigma_{i})$ with $\rho_{A}\circ \widetilde{\phi}=\phi,\rho_{A}\circ\widetilde{\psi}=\widetilde{\psi}.$ By Case 1, we may choose a $p\in S_{d_{i}}$ so that
\[\max_{g\in E} u_{d_{i}}(\{j:p(\sigma_{i}(g)(j))\ne \sigma_{i}(g)(p(j))\})\leq \varepsilon \textnormal{ and } u_{d_{i}}(\{j:\widetilde{\psi}(p(j))\ne \widetilde{\phi}(j)\})\leq \varepsilon.\]
We then have that
\[u_{d_{i}}(\{j:\psi(p(j))\ne \phi(j)\})\leq u_{d_{i}}(\{j:\widetilde{\psi}(p(j))\ne \widetilde{\psi}(j)\})\leq \varepsilon,\]
so this completes the proof.
\end{proof}

Lemma \ref{L:PopaconjationA} automatically tells us that the quantity $\AP(\alpha|\psi,\cdots)$  ``asymptotically does not depend upon $\psi$'' in the case the acting group is amenable. Precisely, we have the following lemma. For the proof, we use the following notation: if $A$ is a finite set, $\delta>0,$ $d\in \N,$ and $\Omega,\Omega'\subseteq A^{d},$ then we write $\Omega\subseteq_{\delta}\Omega'$ if for every $\phi\in \Omega$ there is a $\phi'\in \Omega'$ so that $u_{d}(\{j:\phi(j)\ne \phi'(j)\})\leq \delta.$ This is clearly the same as $\delta$-containment as defined in \cite[Section 3]{Me12} with respect to the metric on $A^{d}$ given by $\rho(\phi,\psi)=u_{d}(\{j:\phi(j)\ne \psi(j)\}).$

\begin{lem}\label{L:asympindependenceAPPa}
Let $\alpha,\beta,\gamma$ be finite measurable observables with $\gamma\geq\alpha\vee \beta.$ Then for any finite $F\subseteq G$ and $\delta\in (0,\frac{1}{2}),$ there exists a finite $F'\subseteq G$ and a $\delta'>0$ so that if $\psi,\psi'\in \AP(\beta:\gamma,F',\delta',\sigma_{i}),$ then
\[|\AP(\alpha|\psi':\gamma,F',\delta',\sigma_{i})|\leq |\AP(\alpha|\psi:\gamma,F,\delta,\sigma_{i})||A|^{\delta d_{i}}\delta d_{i}\binom{d_{i}}{\lfloor{\delta d_{i}\rfloor}}.\]

\end{lem}

\begin{proof}
Let $A,B,C$ be the codomains of $\alpha,\beta,\gamma$ respectively. Choose a $\kappa\in (0,\delta)$ so that for all sufficiently large $i,$ and all $\psi,\psi'\in B^{d_{i}}$ with $u_{d_{i}}(\{j:\psi(j)\ne \psi'(j)\})\leq \kappa$ and $\psi\in \AP(\beta:\gamma,F,\delta/2,\sigma_{i}),$ we have $\psi'\in \AP(\beta:\gamma,F,\delta,\sigma_{i})$ and
	\[\AP(\alpha|\psi:\gamma,F,\delta/2,\sigma_{i})\subseteq_{\kappa} \AP(\alpha|\psi':\gamma,F,\delta,\sigma_{i}).\]
We may choose an $\varepsilon\in (0,\kappa)$ so that if $p\in S_{d_{i}}$ and
\[\max_{g\in F}u_{d_{i}}(\{j:p(\sigma_{i}(g)(j))\ne \sigma_{i}(g)(p(j))\})\leq \varepsilon,\]
then
\[\AP(\gamma,F,\delta/4,\sigma_{i})\circ p\subseteq \AP(\gamma,F,\delta/2,\sigma_{i}).\]
By Lemma \ref{L:PopaconjationA}, we may choose a finite $F'\subseteq G$ with $F'\supseteq F,$ and a $\delta'\in (0,\delta/4)$ so that for all sufficiently large $i$ and all $\psi,\psi'\in \AP(\beta,F',\delta',\sigma_{i}),$ there is a $p\in S_{d_{i}}$ with
\[u_{d_{i}}(\{j:\psi'(p(j))\ne \psi(j)\})\leq \varepsilon\textnormal{ and }\max_{g\in F}u_{d_{i}}(\{j:p(\sigma_{i}(g)(j))\ne \sigma_{i}(g)(p(j))\})\leq \varepsilon.\]
For all sufficiently large $i$ and all $\psi,\psi'\in \AP(\beta:\gamma ,F',\delta',\sigma_{i})$ we may find a $p$ as in Lemma \ref{L:PopaconjationA}. For such a $p$ we have
\[\AP(\alpha|\psi':\gamma,F',\delta',\sigma_{i})\circ p\subseteq \AP(\alpha|\psi'\circ p:\gamma,F,\delta/2,\sigma_{i})\subseteq_{\kappa} \AP(\alpha|\psi:\gamma ,F,\delta,\sigma_{i}).\]
A fortiori,
\[\AP(\alpha|\psi':\gamma,F',\delta',\sigma_{i})\circ p\subseteq_{\delta}  \AP(\alpha|\psi:\gamma ,F,\delta,\sigma_{i}).\]
Since $\delta<1/2$, for a fixed $\phi\in A^{d_{i}}$ we have that $|\{\phi'\in A^{d_{i}}:u_{d_{i}}(\{j:\phi(j)\ne \phi'(j)\})\leq \delta\}|\leq |A|^{d_{i}}\delta d_{i}\binom{d_{i}}{\lfloor{\delta d_{i}\rfloor}}.$ Thus,
\[|\AP(\alpha|\psi':\gamma,F',\delta',\sigma_{i})|=|\AP(\alpha|\psi':\gamma,F',\delta',\sigma_{i})\circ  p|\leq |A|^{d_{i}}\delta d_{i}\binom{d_{i}}{\lfloor{\delta d_{i}\rfloor}}|\AP(\alpha|\psi:\gamma ,F,\delta,\sigma_{i})|.\]

\end{proof}

In order to relate upper relative entropy to relative entropy for amenable groups it turns out to be helpful to write down an equivalent expression for relative entropy.

\begin{defn}\label{D:microstatesrelentA}
 Let $\beta\colon X\to B$ be a finite measurable observable.
We say that a sequence $\psi_{i}\colon\{1,\dots,d_{i}\}\to B$ is a \emph{sequence of $\beta$-microstates} if for every finite $F\subseteq G$ and every $\delta>0$ we have $\psi_{i,\sigma_{i}}^{F}\in \AP(\beta,F,\delta,\sigma_{i})$ for all large $i.$ Suppose that $\alpha,\gamma$ are finite measurable observables with $\gamma\geq \alpha\vee \beta.$  Given a sequence $(\psi_{i})_{i}$ of  $\beta$-microstates, a finite $F\subseteq G,$ and a $\delta>0,$ we set
\[h_{(\sigma_{i})_{i},\mu}(\alpha|(\psi_{i})_{i}:\gamma,F,\delta,\sigma_{i})=\limsup_{i\to\infty}\frac{1}{d_{i}}\log |\AP(\alpha|\psi_{i}:\gamma,F,\delta,\sigma_{i})|,\]
\[h_{(\sigma_{i})_{i},\mu}(\alpha|(\psi_{i})_{i}:\gamma,G)=\inf_{\substack{F\subseteq G \textnormal{finite},\\ \delta>0}}h_{(\sigma_{i})_{i},\mu}(\alpha|(\psi_{i})_{i}:\gamma,F,\delta,\sigma_{i}).\]

\end{defn}

In order to show that the above expressions agree with relative entropy with respect to $(\sigma_{i})_{i}$ as we previously defined we need the following proposition.

\begin{prop}\label{C:amenableMLPaslhg}
Let $\beta,\gamma$ be finite measurable observables with $\beta\leq \gamma.$ Then for every finite $F\subseteq G$ and $\delta>0$ there exists a finite $F'\subseteq G,$  a $\delta'>0,$ and an $I\in \N$ so that
\[\AP(\beta,F',\delta',\sigma_{i})\subseteq \AP(\beta:\gamma,F,\delta,\sigma_{i})\]
for all $i\geq I.$

\end{prop}

\begin{proof} Let $B,C$ be the codomains of $\beta,\gamma.$ Let $\rho\colon C\to B$ be such that $\rho\circ \gamma=\beta$ almost everywhere. Fix a finite $F\subseteq G$ and a $\delta>0.$ We may choose a $\kappa>0$ so that for all $i$ and all $\phi,\psi\in C^{d_{i}}$ with
\[u_{d_{i}}(\{j:\phi(j)\ne \psi(j)\})\leq \kappa, \textnormal{ and } \psi\in \AP(\gamma, F,\delta/2,\sigma_{i}),\]
we have $\phi\in\AP(\gamma,F,\delta,\sigma_{i}).$ We may choose an $\varepsilon\in (0,\kappa)$ so that if $i\in\N,$ and if $p\in S_{d_{i}}$ has
\[\max_{g\in F}u_{d_{i}}(\{j:p(\sigma_{i}(g)(j))\ne \sigma_{i}(g)p(j)\})\leq \varepsilon,\]
then
\[\AP(\gamma,F,\delta/4,\sigma_{i})\circ p\subseteq \AP(\gamma,F,\delta/2,\sigma_{i}).\]

By Lemma \ref{L:PopaconjationA}, we may choose a finite $F'\subseteq G,$ a $\delta'>0,$ and an $I_{0}\in \N$ so that if $i\geq I_{0},$ then for all $\phi,\psi\in\AP(\beta,F',\delta',\sigma_{i})$ there is a $p\in S_{d_{i}}$ with
\[u_{d_{k}}(\{j:\psi(p(j))\ne \phi(j)\})\leq \varepsilon \textnormal{ and  } \max_{g\in F}u_{d_{i}}(j:p(\sigma_{i}(g)(j))\ne \sigma_{i}(g)(p(j))\})\leq \varepsilon.\]
We may, and will, assume that $F\subseteq F'$ and that $\delta'<\delta/4.$

Since $G$ is amenable, it follows from Theorems 1 and 4 of \cite{ElekLip} (see also Theorem 1 of \cite{BowenAmen}, Theorem 6.7 of \cite{KLi2})  that we may find an $I_{1}\in \N$ so that if $i\geq I_{1},$  then there is a $\psi\in \AP(\gamma,F',\delta',\sigma_{i}).$ Fix an $i\geq \max(I_{0},I_{1})$, a $\psi\in \AP(\gamma,F',\delta,',\sigma_{i}),$ and  a $\phi\in\AP(\beta,F',\delta',\sigma_{i}).$ Since $\rho\circ \psi\in \AP(\beta,F',\delta',\sigma_{i})$, we may find a $p\in S_{d_{i}}$ so that
\[u_{d_{i}}(\{j:(\rho\circ \psi)(p(j))\ne \phi(j)\})\leq \varepsilon.\]
Let $J=\{j:(\rho\circ \psi)(p(j))\ne \phi(j)\}.$ Let $\widetilde{\phi}\in C^{d_{i}}$ be any function satisfying the following two conditions:
\begin{itemize}
\item $\widetilde{\phi}\big|_{J^{c}}=(\psi\circ p)\big|_{J^{c}}$,
\item $(\rho\circ \widetilde{\phi})\big|_{J}=\phi\big|_{J}$.
\end{itemize}
By our choice of $\delta',F',\varepsilon,$ we have that $\psi\circ p\in \AP(\gamma,F,\delta/2,\sigma_{i}).$ Since
\[u_{d_{i}}(\{j:\widetilde{\phi}(j)\ne \psi(p(j))\})=u_{d_{i}}(J)\leq \varepsilon\leq \kappa,\]
it follows from our choice of $\kappa$ that $\widetilde{\phi}\in \AP(\gamma,F,\delta,\sigma_{i}).$ Since $\rho\circ \widetilde{\phi}=\phi$ by construction,  we have shown that
\[\AP(\beta,F',\delta',\sigma_{i})\subseteq \AP(\beta:\gamma,F,\delta,\sigma_{i}).\]

\end{proof}

The above lemma shows that any factor map between actions of an amenable group is \emph{model-surjective} in the sense of \cite[Definition 3.1]{AustinGeo}. Combining this with Proposition 8.4 of \cite{AustinAdd} gives an alternate proof of the fact that measure-preserving actions of amenable groups are strongly sofic with respect to any sofic approximation.

\begin{lem}\label{L:infsupdontmatter}
Let $\beta,\alpha,\gamma$ be  finite measurable observables with domain $X$ and with $\alpha\vee \beta\leq \gamma$. If $\psi_{i}$ is any sequence of $\beta$-microstates, then
\[h_{(\sigma_{i})_{i},\mu}(\alpha|\beta:\gamma)=h_{(\sigma_{i})_{i},\mu}(\alpha|(\psi_{i})_{i}:\gamma,G).\]
\end{lem}

\begin{proof}
Fix a finite $F\subseteq G$ and a $\delta>0.$ If $i$ is sufficiently large, then by Proposition \ref{C:amenableMLPaslhg} we have $\psi_{i}\in \AP(\beta:\gamma,F,\delta,\sigma_{i}).$ Thus,
\[h_{(\sigma_{i})_{i},\mu}(\alpha|\beta:\gamma)\geq h_{(\sigma_{i})_{i},\mu}(\alpha|(\psi_{i})_{i}:\gamma,G).\]

	We now  prove the reverse inequality. Fix a finite $F\subseteq G$ and a $\delta>0.$ Choose $F',\delta'$ as in Lemma \ref{L:asympindependenceAPPa} for this $F,\delta.$
Then for all large $i,$
\[\sup_{\phi\in \AP(\alpha:\gamma,F',\delta',\sigma_{i})}|\AP(\alpha|\phi:\gamma,F',\delta',\sigma_{i})|\leq |\AP(\alpha|\psi_{i}:\gamma,F,\delta,\sigma_{i})||A|^{\delta d_{i}}\delta d_{i}\binom{d_{i}}{\lfloor{\delta d_{i}\rfloor}}.\]
By Stirling's Formula,
\begin{align*}
h_{(\sigma_{i})_{i},\mu}(\alpha|\beta:\gamma)&\leq h_{(\sigma_{i})_{i},\mu}(\alpha|\beta:\gamma,F',\delta',\sigma_{i})\\
&\leq h_{(\sigma_{i})_{i},\mu}(\alpha|(\psi_{i})_{i}:\gamma,F,\delta)-\delta \log(\delta)-(1-\delta)\log(1-\delta)+\delta \log|A|.
\end{align*}
Letting $\delta\to 0$ and then taking the infimum over all finite $F\subseteq G$ completes the proof.

\end{proof}

It will also be helpful to note that we may replace the limit supremum in the definition of relative entropy with a limit infimum.

\begin{defn}\label{D:microstatesrelentB}
Let $\alpha,\gamma$ be finite measurable observables with $\gamma\geq \alpha.$  Given a finite $F\subseteq G$ and a $\delta>0,$ we set
\[\underline{h}_{(\sigma_{i})_{i},\mu}(\alpha:\gamma,F,\delta,\sigma_{i})=\liminf_{i\to\infty}\frac{1}{d_{i}}\log |\AP(\alpha:\gamma,F,\delta,\sigma_{i})|,\]
\[\underline{h}_{(\sigma_{i})_{i},\mu}(\alpha:\gamma)=\inf_{\substack{F\subseteq G \textnormal{finite},\\ \delta>0}}\underline{h}_{(\sigma_{i})_{i},\mu}(\alpha:\gamma,F,\delta,\sigma_{i}).\]

\end{defn}

The following is more or less a consequence of results of Bowen \cite{BowenAmen} and Kerr-Li \cite{KLi2}

\begin{lem}\label{L:liminflimsupdon'tmatter}
Let $\alpha,\gamma$ be finite measurable observables with domain $X$ and so that $\alpha\leq \gamma.$ Then
\[\underline{h}_{(\sigma_{i})_{i},\mu}(\alpha:\gamma)=h_{\mu}(\alpha,G)=h_{(\sigma_{i})_{i},\mu}(\alpha:\gamma).\]

\end{lem}

\begin{proof}  It is implicitly shown in \cite{BowenAmen},\cite{KLi2} that
\[h_{(\sigma_{i})_{i},\mu}(\alpha:\alpha)=h_{\mu}(\alpha,G)=\underline{h}_{(\sigma_{i})_{i},\mu}(\alpha:\alpha),\]
Hence,
\[h_{\mu}(\alpha,G)=h_{(\sigma_{i})_{i}}(\alpha:\alpha)\geq h_{(\sigma_{i})_{i}}(\alpha:\gamma)\geq \underline{h}_{(\sigma_{i})_{i}}(\alpha:\gamma).\]
It thus suffices to show that $h_{\mu}(\alpha,G)\leq \underline{h}_{(\sigma_{i})_{i}}(\alpha:\gamma).$ To show this, fix a finite $F\subseteq G$ and a $\delta>0.$ By Proposition \ref{C:amenableMLPaslhg}, we may find a finite $F'\subseteq G$ and a $\delta'>0$ so that
\[\AP(\alpha:\alpha,F',\delta',\sigma_{i})\subseteq \AP(\alpha:\gamma,F,\delta,\sigma_{i}).\]
It follows that
\[h_{\mu}(\alpha,G)=\underline{h}_{(\sigma_{i})_{i},\mu}(\alpha:\alpha)\leq \underline{h}_{(\sigma_{i})_{i},\mu}(\alpha:\alpha,F',\delta')\leq \underline{h}_{(\sigma_{i})_{i},\mu}(\alpha:\gamma,F,\delta).\]
And taking the infimum over $F,\delta$ completes the proof.

\end{proof}

\begin{thm}\label{T:Appendixagreementfinitecase}
Fix finite measurable observables $\beta,\alpha,\gamma$ with $\gamma\geq \alpha\vee \beta.$ Let $S_{\beta}$ be the smallest complete, $G$-invariant sub-sigma-algebra of $\mathcal{X}$ which makes $\beta$ measurable. Then
\[h_{(\sigma_{i})_{i},\mu}(\alpha|\beta:\gamma)=h_{\mu}(\alpha|S_{\beta},G).\]
\end{thm}

\begin{proof}
 By Proposition \ref{P:basicpropertieslakdhgkl} (\ref{I:bound2alkghal}) and Lemma \ref{L:liminflimsupdon'tmatter},
\begin{align*}
h_{\mu}(\alpha\vee \beta,G)=h_{(\sigma_{i})_{i},\mu}(\alpha\vee \beta:\gamma)&\leq h_{(\sigma_{i})_{i},\mu}(\beta:\gamma)+h_{(\sigma_{i})_{i},\mu}(\alpha|\beta:\gamma)\\
&= h_{\mu}(\beta,G)+h_{(\sigma_{i})_{i},\mu}(\alpha|\beta:\gamma).
\end{align*}
Using the Abramov-Rokhlin formula $h_{\mu}(\alpha|S_{\beta},G)=h_{\mu}(\alpha\vee \beta,G)-h_{\mu}(\beta,G)$ (see \cite{AbramovRohklin,AbramovRokhlin2,AbramovRohklin3,Danilenko}),    we find that
\[h_{\mu}(\alpha|S_{\beta},G)\leq h_{(\sigma_{i})_{i},\mu}(\alpha|\beta:\gamma).\]

To prove the reverse inequality, choose a sequence $(\psi_{i})_{i}$ of $\beta$-microstates. Fix a finite $F\subseteq G$ and a $\delta>0,$ and let $F'\subseteq G$ and $\delta'>0$ be as in the conclusion to Lemma \ref{L:asympindependenceAPPa}.
For all sufficiently large $i,$
\begin{align*}
|\AP(\alpha\vee \beta :\gamma,F,\delta,\sigma_{i})|&=\sum_{\phi\in \AP(\beta:\gamma,F,\delta,\sigma_{i})}|\AP(\alpha|\phi,F,\delta,\sigma_{i})|\\
&\geq \sum_{\phi\in \AP(\beta:\gamma,F',\delta',\sigma_{i})}|\AP(\alpha|\phi,F,\delta,\sigma_{i})|\\
&\geq |\AP(\beta:\gamma,F',\delta',\sigma_{i})||\AP(\alpha|\psi_{i}:F',\delta',\sigma_{i})||A|^{-\delta d_{i}}\frac{1}{\delta d_{i}\binom{d_{i}}{\lfloor{\delta d_{i}\rfloor}}}.
\end{align*}
So  by Stirling's formula,
\begin{align*}
h_{(\sigma_{i})_{i},\mu}(\alpha\vee \beta:\gamma,F,\delta)&\geq h_{(\sigma_{i})_{i},\mu}(\alpha|(\psi_{i})_{i}:F',\delta')+\underline{h}_{(\sigma_{i})_{i},\mu}(\beta:\gamma,F',\delta')\\
&-\delta\log|A|+\delta \log(\delta)+(1-\delta)\log(1-\delta)\\
&\geq h_{(\sigma_{i})_{i},\mu}(\alpha|\beta:\gamma)+h_{\mu}(\beta,G)-\delta\log|A|+\delta \log(\delta)+(1-\delta)\log(1-\delta),
\end{align*}
where in the last line we use Lemmas \ref{L:liminflimsupdon'tmatter} and \ref{L:infsupdontmatter}. Taking the infimum over all $F,\delta$ and arguing as in the first half we see that
\[h_{(\sigma_{i})_{i},\mu}(\alpha|\beta:\gamma)\leq h_{\mu}(\alpha|S_{\beta},G).\]

\end{proof}

\begin{cor}
For any $G$-invariant sigma algebras $\mathcal{F}_{1},\mathcal{F}_{2}\subseteq \mathcal{X}$ we have
\[h_{\mu}(\mathcal{F}_{1}|\mathcal{F}_{2},G)=h_{(\sigma_{i})_{i},\mu}(\mathcal{F}_{1}|\mathcal{F}_{2}:\mathcal{X}).\]
\end{cor}

\begin{proof} For a finite $\mathcal{F}_{2}$-measurable observable $\beta,$ let $S_{\beta}$ be defined as in Theorem \ref{T:Appendixagreementfinitecase}. We have that
\[h_{\mu}(\mathcal{F}_{1}|\mathcal{F}_{2},G)=\sup_{\alpha}\inf_{\beta}h_{\mu}(\alpha|S_{\beta},G),\]
where the supremum is over all finite $\mathcal{F}_{1}$-measurable observables $\alpha,$ and the infimum is over all finite $\mathcal{F}_{2}$-measurable observables $\beta.$ A similar formula also holds for $h_{(\sigma_{i})_{i},\mu}(\mathcal{F}_{1}|\mathcal{F}_{2}:\mathcal{X}),$ so the general case follows from Theorem \ref{T:Appendixagreementfinitecase}.

\end{proof}

%
%

\end{document}